\documentclass[11pt]{amsart}
\usepackage{times}
\usepackage[pdftex]{graphicx}
\usepackage{amssymb,amsfonts,amsmath,amsthm}
\usepackage{float}

\newtheorem{theorem}{Theorem}
\newtheorem{proposition}{Proposition}
\newtheorem{lemma}{Lemma}
\newtheorem{definition}{Definition}
\newtheorem{remark}{Remark}
\newtheorem{corollary}{Corollary}

\newcommand{\EE} {\mathbb E}
\newcommand{\HH} {\mathbb H}
\newcommand{\PP} {\mathbb P}
\newcommand{\RR}{\mathbb R}
\newcommand{\bA} {\mathbb A}
\newcommand{\cP}{\mathcal P}
\def\balpha{\bar \alpha}
\def\bX{\bar X}

\begin{document}

\title{Probabilistic Analysis of Mean-Field Games}

\author{Ren\'e Carmona}
\address{ORFE, Bendheim Center for Finance, Princeton University,
Princeton, NJ  08544, USA.}
\email{rcarmona@princeton.edu}
\thanks{Partially supported  by NSF: DMS-0806591}

\author{Francois Delarue}
\address{Laboratoire Jean-Alexandre Dieudonn\'e 
UniversitŽ de Nice Sophia-Antipolis 
Parc Valrose 
06108 Cedex 02, Nice, FRANCE}
\email{Francois.Delarue@unice.fr}

\subjclass[2000]{Primary }

\keywords{}

\date{June 10, 2011}

\begin{abstract}
The purpose of this paper is to provide a complete probabilistic analysis of a large class of stochastic differential games 
for which the interaction  between the players is of mean-field type. 
We implement the Mean-Field Games strategy developed analytically by Lasry and Lions in a purely probabilistic
framework, relying on tailor-made forms of the stochastic maximum principle. While we assume that the state dynamics are affine
in the states and the controls, our assumptions on the nature of the costs are rather weak, and surprisingly,
the dependence of all the coefficients upon the statistical distribution of the states remains of a rather general nature.  
Our probabilistic approach calls for the solution of systems of forward-backward stochastic differential equations of a McKean-Vlasov type for which no existence result is known, and for which we prove existence and regularity of the corresponding value function. Finally, we prove that solutions of the mean-field game as formulated by Lasry and Lions 
do indeed provide approximate Nash equilibriums for games with a large number of players, and we quantify the nature of the approximation. 
\end{abstract}

\maketitle

\section{\textbf{Introduction}}
In a trailblazing contribution, Lasry and Lions \cite{MFG1,MFG2,MFG3} proposed a methodology to produce approximate Nash equilibriums for stochastic differential games  with symmetric interactions and a large number of players. In their model, the costs to a given player \emph{feel} the presence and the behavior of the other players through the empirical distribution of their private states. This type of interaction was introduced and
studied in statistical physics under the name of \emph{mean-field} interaction, allowing for the derivation of effective equations in the limit of asymptotically large systems.   Using intuition and mathematical results from propagation of chaos, Lasry and Lions propose to assign to each player, independently of what other players may do, a distributed closed loop strategy given by the solution of the limiting problem, arguing that such a resulting game should be in an \emph{approximate} Nash equilibrium. This streamlined approach is very attractive as large stochastic differential games are notoriously nontractable. They formulated the limiting problem as a system of two highly coupled nonlinear partial differential equations (PDE for short): the first one, of the Hamilton-Jacobi-Bellman type, takes care of the optimization part, while the second one, of Kolmogorov type, guarantees the time consistency of the statistical distributions of the private states of the individual players. The issue of existence and uniqueness of solutions for such a system is a very delicate problem, as the solution of the former equation should propagate backward in time from a terminal condition while the solution of the latter should evolve forward in time from an initial condition. More than the nonlinearities,  the conflicting directions of time compound the difficulties.

In a subsequent series of works \cite{GiraudGueantLasryLions,FDD09,GueantLasryLions.pplnm,Lachapelle,LachapelleLasry} with PhD students and postdoctoral fellows, Lasry and Lions considered applications to domains as diverse as the management of exhaustible resources like oil, house insulation, and the analysis of pedestrian crowds. Motivated by problems in large communication networks, Caines, Huang and Malham\'e introduced, essentially at the same time \cite{HuangCainesMalhame2}, a similar strategy which they call the Nash Certainty Equivalence. They also studied practical applications to large populations behavior \cite{HuangCainesMalhame1}.

\vskip 2pt
The goal of the present paper is to study the effective Mean-Field Game equations proposed by Lasry and Lions, from a probabilistic point of view. To this end, we recast the challenge as a fixed point problem in a space of flows of probability measures, show that these fixed points do exist and provide approximate Nash equilibriums for large games, and quantify the accuracy of the approximation. 

We tackle the limiting stochastic optimization problems
using the probabilistic approach of the stochastic maximum principle, thus reducing the problems to the solutions of Forward Backward 
Stochastic Differential Equations (FBSDEs for short). The search for a fixed flow of probability measures turns the system of forward-backward stochastic differential equations into equations of the McKean-Vlasov type where the distribution of the solution appears
in the coefficients. In this way, both the optimization and interaction components of the problem are captured by a single FBSDE, avoiding the twofold reference to Hamilton-Jacobi-Bellman equations on the one hand, and Kolmogorov equations on the other hand. As a by-product of this approach, the stochastic dynamics of the states could be degenerate. We give a general overview of this strategy in Section \ref{se:notation} below. Motivated in part by the works of Lasry, Lions and collaborators, Backward Stochastic Differential Equations (BSDEs) of the mean field type have recently been studied. See for example \cite{BuckdahnDjehicheLi1,BuckdahnDjehicheLiPeng}. However, existence and uniqueness results for BSDEs are much easier to come by than for FBSDEs, and here, we have to develop existence results from scratch. 

Our first existence result is proven for bounded coefficients by means of a fixed point argument based on Schauder's theorem pretty much in the same spirit as Cardaliaguet's notes \cite{Cardaliaguet}. Unfortunately, such a result does not apply to some of the linear-quadratic (LQ) games already studied  \cite{HuangCainesMalhame4,Bardi,Bensoussanetal,CarmonaDelarueLachapelle}, and some of the most technical proofs of the papers are devoted to the extension of this existence result to coefficients with linear growth. See Section \ref{sec:3}. 
Our approximation and convergence arguments are based on probabilistic \emph{a priori} estimates obtained from tailor-made versions of the stochastic maximum principle which we derive in Section \ref{se:notation}. The reader is referred to the book of Ma and Yong \cite{MaYong.book} for background material on adjoint equations, FBSDEs and the stochastic maximum principle approach to stochastic optimization problems. As we rely on this approach, we find it natural to derive the compactness properties needed in our proofs from convexity properties of the coefficients of the game. The reader is also referred to the papers by Hu and Peng \cite{HuPeng} and Peng and Wu \cite{PengWu} for general solvability properties of standard FBSDEs within the same framework of stochastic optimization.

The thrust of our analysis is not limited to existence
of a solution to a rather general class of McKean-Vlasov FBSDEs, but also to the extension to this non-Markovian set-up of the construction of the
FBSDE value function expressing the solution of the backward equation in terms of the solution of the forward dynamics. The existence of this value function is crucial for the formulation and the proofs of the results of the last part of the paper. In Section \ref{se:apNash}, we indeed prove that the solutions of the \emph{fixed point FBSDE} (which include a function $\hat\alpha$ minimizing the Hamiltonian of the system, three stochastic processes $(X_t,Y_t,Z_t)_{0\le t\le T}$ solving the FBSDE, and the FBSDE value function $u$)  provide a set of distributed strategies which, when used by the players of a $N$-player game, form an $\epsilon_N$-approximate Nash equilibrium, and we quantify the speed at which $\epsilon_N$ tends to $0$ when $N\to+\infty$. This type of argument has been used for simpler models in \cite{Bensoussanetal} or \cite{Cardaliaguet}. Here, we use convergence estimates which are part of the standard theory of propagation of chaos (see for example \cite{Sznitman,JourdainMeleardWoyczynski}) and the Lipschitz continuity and linear growth the FBSDE value function $u$ which we prove earlier in the paper.

\section{\textbf{General Notation and Assumptions}}
\label{se:notation}
Here, we introduce the notation and the basic tools from stochastic analysis which we use throughout the paper.

\subsection{The $N$ Player Game}
We consider a stochastic differential game with $N$ players, each player $i\in\{1,\cdots,N\}$ controlling his own private state $U^i_t\in\RR^d$ at time $t\in[0,T]$ 
by taking an action $\beta^i_t$ in a set $A\subset\RR^k$.
We assume that the dynamics of the private states of the individual players are given by It\^o's stochastic differential equations of the form
\begin{equation}
\label{fo:private_dynamics}
dU^i_t=b^i(t,U^i_t,{\bar{\nu}}_t^N,\beta^i_t)dt+\sigma^i(t,U^i_t,{\bar{\nu}}^N_t,\beta^i_t)dW^i_t, \qquad 0\le t\le T, 
\quad i=1,\cdots,N,\\
\end{equation}
where the $W^i=(W^i_t)_{0\le t\le T}$ are $m$-dimensional independent Wiener processes, 
$
(b^i,\sigma^i):[0,T]\times \RR^d\times \cP(\RR^d)\times A\hookrightarrow \RR^{d} \times \RR^{d \times m}
$
are deterministic measurable functions satisfying a set of assumptions spelled out below, and $\bar{\nu}_t^N$ denotes the empirical distribution of $U_t=(U^1_t,\cdots,U^N_t)$ defined as
$$
{\bar{\nu}}_t^N(dx')=\frac1N\sum_{i=1}^N\delta_{U_t^i}(dx').
$$
Here and in the following, we use the notation $\delta_x$ for the Dirac measure (unit point mass) at $x$, and $\cP(E)$ for the space of probability measures on $E$
whenever $E$ is a topological space equipped with its Borel $\sigma$-field. In this framework, ${\mathcal P}(E)$ itself is endowed with the 
Borel $\sigma$-field generated by the topology of weak convergence of measures. 
 \vskip 2pt
 Each player chooses a strategy in the space $\bA=\HH^{2,k}$ of progressively measurable $A$-valued stochastic processes 
 $\beta=(\beta_t)_{0\le t\le T}$ satisfying the admissibility condition:
\begin{equation}
\label{eq:4:4:1}
\EE \biggl[ \int_{0}^T \vert \beta_{t} \vert^2 dt \biggr] < + \infty.
\end{equation}
The choice of a strategy is driven by the desire to minimize an expected cost over the period $[0,T]$, each individual cost being a combination of running and terminal costs. For each $i\in\{1,\cdots,N\}$, the running cost to player $i$ is given by a measurable function $f^i:[0,T]\times\RR^d\times \cP(\RR^d)\times A\hookrightarrow \RR$ and the terminal cost by a measurable function $g^i:\RR^d\times\cP(\RR^d) \hookrightarrow \RR$ in such a way that if the $N$ players use the strategy $\beta=(\beta^1,\cdots,\beta^N)\in\bA^N$, the expected total cost to player $i$ is 
\begin{equation}
\label{fo:cost}
J^i(\beta)=\EE\biggl[g^i(U^i_T,{\bar{\nu}}^N_T) + \int_0^Tf^i\bigl(t,U^i_t,{\bar{\nu}}_t^N,\beta^i_t\bigr)dt \biggr].
\end{equation}
Here $\bA^N$ denotes the product of $N$ copies of $\bA$.  Later in the paper, we let $N\to \infty$ and use the notation $J^{N,i}$ in order to emphasize the dependence upon $N$. Notice that even though only $\beta^i_t$ appears in the formula giving the cost to player $i$, this cost depends upon the strategies used by the other players indirectly, as these strategies affect not only the private state $U_t^i$, but also the empirical distribution $\bar{\nu}^N_t$ of all the private states. As explained in the introduction, our model requires that the behaviors of the players be \textit{statistically identical}, imposing that the coefficients $b^i$, $\sigma^i$, $f^i$ and $g^i$ do not depend upon $i$. We denote them by $b$, $\sigma$, $f$ and $g$.

\vskip 2pt
In solving the game, we are interested in the notion of optimality given by the concept of Nash equilibrium.
Recall that a set of admissible strategies $\alpha^*=(\alpha^{*1},\cdots,\alpha^{*N})\in\bA^N$ is said to be a Nash equilibrium for the game if 
$$
\forall i\in\{1,\cdots,N\},\forall \alpha^i\in\bA,\qquad J^i(\alpha^*)\le J^i(\alpha^{*-i},\alpha^i).
$$
where we use the standard notation $(\alpha^{*-i},\alpha^i)$ for the set of strategies $(\alpha^{*1},\cdots,\alpha^{*N})$ where $\alpha^{*i}$ has been replaced by $\alpha^i$.

\subsection{The Mean-Field Problem}
\label{subsec:2:2}
In the case of large symmetric games, some form of averaging is expected when the number of players tends to infinity. The Mean-Field Game (MFG) philosophy of Lasry and Lions is to search for approximate Nash equilibriums  through the solution of effective equations appearing in the limiting regime $N\to\infty$, and assigning to each player the strategy $\alpha$ provided by the solution of the effective system of equations they derive. In the present context, the implementation of this idea involves the solution of the following fixed point problem which we break down in three steps for pedagogical reasons:
\begin{enumerate}\itemsep=-2pt
\item[(i)] Fix a deterministic function $[0,T]\ni t\hookrightarrow \mu_t\in\cP(\RR^d)$;
\item[(ii)] Solve the standard stochastic control problem
\begin{equation}
\label{fo:mfgcontrolpb}
\begin{split}
&\inf_{\alpha\in\bA}\EE\left[\int_0^Tf(t,X_t,\mu_t,\alpha_t)dt + g(X_T,\mu_T)\right]
\\
&\mbox{subject to} \quad dX_t=b(t,X_t,\mu_t,\alpha_t)dt + \sigma(t,X_t,\mu_t,\alpha_t) dW_t; \quad X_{0}=x_{0}.
\end{split}
\end{equation}
\item[(iii)] Determine the function $[0,T]\ni t\hookrightarrow \mu_t\in\cP(\RR^d)$  so that
$
\forall t\in[0,T],\  \PP_{X_t}=\mu_t.$
\end{enumerate}
Once these three steps have been taken successfully, if the \textit{fixed-point} optimal control $\alpha$ identified in step (ii) is in feedback form,  i.e. of the form $\alpha_t=\hat\alpha(t,X_t,\PP_{X_t})$
for some function $\hat\alpha$ on $[0,T]\times \RR^d\times\cP(\RR^d)$, denoting by $\hat\mu_t=\PP_{X_t}$ the \textit{fixed-point} marginal distributions, the prescription 
$\hat{\alpha}^{i*}_t=\hat\alpha(t,X^i_t,\hat\mu_t)$, if used by the players $i=1,\cdots,N$ of a large game, should form an approximate Nash equilibrium.
We prove this fact rigorously in Section \ref{se:apNash} below, and we quantify the accuracy of the approximation.

\subsection{The Hamiltonian} 
\label{subsec:2:3}
For the sake of simplicity, we assume that $A=\RR^k$, and in order to lighten the notation and to avoid many technicalities, that the volatility is an uncontrolled constant matrix $\sigma  \in \RR^{d \times m}$.
The fact that the volatility is uncontrolled allows us to use the simplified version for the Hamiltonian:
\begin{equation}
\label{fo:Hamiltonian}
H(t,x,\mu,y,\alpha) = \langle b(t,x,\mu,\alpha),y \rangle + f(t,x,\mu,\alpha),
\end{equation}
for $t \in [0,T]$, $x,y \in \RR^d$, $\alpha\in\RR^k$, and $\mu \in {\mathcal P}(\RR^d)$. Our first task will be to minimize the Hamiltonian with respect to the control parameter, and understand how minimizers depend upon the other variables. 
We shall use the following standing assumptions.
\vskip 2pt\noindent
\textbf{(A.1)} The drift $b$ is an affine function of $\alpha$ in the sense that it is of the form
\begin{equation}
\label{fo:drift1}
b(t,x,\mu,\alpha)=b_1(t,x,\mu) +b_2(t)\alpha,
\end{equation}
where the mapping $[0,T] \ni t \hookrightarrow b_2(t) \in \RR^{d \times k}$ is measurable and bounded,
and the mapping $[0,T] \ni (t,x,\mu) \hookrightarrow b_{1}(t,x,\mu) \in \RR^d$ is measurable and bounded on bounded subsets of $[0,T] \times \RR^d \times {\mathcal P}_{2}(\RR^d)$. 

\vskip 2pt
Here and in the following, whenever $E$ is a separable Banach space and $p$ is an integer greater than $1$, $\cP_p(E)$ stands for the subspace of $\cP(E)$ of probability measures of order $p$, i.e. having a finite moment of order $p$ so that $\mu \in \cP_p(E)$ if 
$\mu \in \cP(E)$ and 
\begin{equation}
\label{eq:20:6:1}
M_{p,E}(\mu) = \biggl(\int_E  \|x\|_{E}^p d\mu(x) \biggr)^{1/p} < + \infty.
\end{equation}
We write $M_{p}$ for $M_{p,\RR^d}$.
Below, bounded subsets of ${\mathcal P}_{p}(E)$ are defined as sets of probability measures with uniformly bounded moments of order p. 
\vskip 2pt\noindent
\textbf{(A.2)}  There exist two positive constants $\lambda$ and $c_{L}$ such that for any $t \in [0,T]$ and $\mu \in {\mathcal P}_{2}(\RR^d)$, the function 
$\RR^d \times \RR^k \ni (x,\alpha)\hookrightarrow f(t,x,\mu,\alpha)\in\RR$ is once continuously differentiable with Lipschitz-continuous derivatives (so that $f(t,\cdot,\mu,\cdot)$ is $C^{1,1}$), the Lipschitz constant in $x$ and $\alpha$ being bounded by 
$c_L$ (so that it is uniform in $t$ and $\mu$). Moreover, 
it satisfies the convexity assumption 
\begin{equation}
\label{fo:lambdaconvexity}
f(t,x',\mu,\alpha') - f(t,x,\mu,\alpha) - \langle (x'-x,\alpha'-\alpha),\partial_{(x,\alpha)} f(t,x,\mu,\alpha) \rangle \geq \lambda 
\vert \alpha' - \alpha \vert^2. 
\end{equation} 
The notation $\partial_{(x,\alpha)} f$ stands for the gradient in the joint variables $(x,\alpha)$. Finally, $f$, $\partial_{x} f$ and $\partial_{\alpha} f$ are locally bouded over $[0,T] \times \RR^d \times 
{\mathcal P}_{2}(\RR^d) \times \RR^k$. 
\vskip 4pt
The minimization of the Hamiltonian is taken care of by the following result.
\begin{lemma}
\label{le:minimizer}
If we assume that assumptions (A.1--2) are in force, then, for all $(t,x,\mu,y)
\in [0,T] \times \RR^d \times {\mathcal P}_{2}(\RR^d) \times \RR^k$, 
there exists a unique minimizer $\hat{\alpha}(t,x,\mu,y)$ of $H$.
Moreover, the function $[0,T] \times \RR^d \times {\mathcal P}_{2}(\RR^d) \times \RR^d
\ni (t,x,\mu,y)   \hookrightarrow \hat{\alpha}(t,x,\mu,y)$ is measurable, locally bounded and Lipschitz-continuous with respect to 
$(x,y)$, uniformly in $(t,\mu) \in [0,T] \times {\mathcal P}_{2}(\RR^d)$, the Lipschitz constant depending only upon $\lambda$, the supremum norm 
of $b_{2}$ and the Lipschitz 
constant of $\partial_{\alpha}f$ in $x$.
\end{lemma}

\begin{proof}
For any given $(t,x,\mu,y)$, the function $\RR^k \ni \alpha \hookrightarrow H(t,x,\mu,y,\alpha)$ is once continuously differentiable and strictly convex so that $\hat{\alpha}(t,x,\mu,y)$ appears as the unique solution of the equation
$\partial_{\alpha} H(t,x,\mu,y,\hat{\alpha}(t,x,\mu,y)) = 0.$
By strict convexity, measurability of the minimizer $\hat{\alpha}(t,x,\mu,y)$ is a consequence of the gradient descent algorithm. Local boundedness of $\hat{\alpha}(t,x,\mu,y)$ also follows from strict convexity since by \eqref{fo:lambdaconvexity},
\begin{equation*}
\begin{split}
H(t,x,\mu,y,0) &\geq  H(t,x,\mu,y,\hat{\alpha}(t,x,\mu,y)\bigr)\\
&\geq H(t,x,\mu,y,0) + \langle \hat{\alpha}(t,x,\mu,y) ,\partial_{\alpha} H(t,x,\mu,y,0) \rangle
 + \lambda \bigl\vert  \hat{\alpha}(t,x,\mu,y) \bigr\vert^2, 
\end{split}
\end{equation*}
so that 
\begin{equation}
\label{eq:3:5:1}
\bigl\vert \hat{\alpha}(t,x,\mu,y) \bigr\vert \leq \lambda^{-1} \bigl( \vert \partial_{\alpha}
f(t,x,\mu,0) \vert + \vert b_{2}(t) \vert \, \vert y \vert \bigr).
\end{equation} 
Inequality \eqref{eq:3:5:1} will be used repeatedly. Moreover,  by the implicit function theorem, $\hat\alpha$ is 
Lipschitz-continuous with respect to $(x,y)$, the Lipschitz-constant being controlled by 
the uniform bound on $b_{2}$ and by the Lipschitz-constant of $\partial_{(x,\alpha)} f$. 
\end{proof}

\subsection{Stochastic Maximum Principle} 
Going back to the program (i)--(iii) outlined in Subsection \ref{subsec:2:2}, the first two steps therein consist in 
solving a standard minimization problem when the distributions $(\mu_{t})_{0 \leq t \leq T}$ are \textit{frozen}, and one could express the value function of the optimization problem \eqref{fo:mfgcontrolpb} as the solution of the corresponding Hamilton-Jacobi-Bellman (HJB for short) equation. This is the keystone of the analytic approach to the MFG theory, the matching problem 
(iii) being resolved by coupling the HJB equation with a Kolmogorov equation intended to identify the  $(\mu_{t})_{0 \leq t \leq T}$ with the marginal distributions of the optimal state of the problem. 

Instead, the strategy we have in mind relies on a probabilistic description of the optimal states of the optimization problem \eqref{fo:mfgcontrolpb} as provided by the so-called stochastic maximum principle. Indeed, the latter provides a necessary condition for the optimal states of the problem \eqref{fo:mfgcontrolpb}: under suitable conditions, the optimally controlled diffusion processes satisfy the forward dynamics in a characteristic FBSDE, referred to as \emph{the adjoint system} of the stochastic optimization problem. Moreover, the stochastic maximum principle provides a sufficient condition since, under additional convexity conditions, the forward dynamics of any solution to the adjoint system are optimal. In what follows, we use the sufficiency condition for proving the existence of solutions to the limit problem (i)--(iii) stated  in Subsection \ref{subsec:2:2}.  In addition to (A.1--2) we will also assume:
\vskip 2pt\noindent
\textbf{(A.3)} The function $[0,T] \ni t \hookrightarrow b_{1}(t,x,\mu)$ is affine in $x$, i.e. it has the form
$[0,T] \ni t \hookrightarrow b_{0}(t,\mu) + b_{1}(t)x$, where $b_{0}$ and $b_{1}$ are $\RR^d$ and 
$\RR^{d \times d}$ valued respectively, and are bounded on bounded subsets of their respective domains. In particular, $b$ reads
\begin{equation}
\label{fo:drift}
b(t,x,\mu,\alpha) = b_{0}(t,\mu) + b_{1}(t) x + b_{2}(t) \alpha.
\end{equation}
\vskip 2pt
\noindent
\textbf{(A.4)} The function $\RR^d \times {\mathcal P}_{2}(\RR^d) \ni (x,\mu) 
\hookrightarrow g(x,\mu)$ is locally bounded. Moreover, for any $\mu \in {\mathcal P}_{2}(\RR^d)$, the function 
$\RR^d \ni x \hookrightarrow g(x,\mu)$ is once continuously differentiable and convex and has a $c_{L}$-Lipschitz-continuous first order derivative. 
\vskip 4pt
In order to make the paper self-contained, we state and briefly prove the form of the sufficiency part of the stochastic maximum principle 
as it applies to (ii) when the flow of measures $(\mu_{t})_{0 \leq t \leq T}$ are frozen. Instead of the standard version given for example
in Chapter IV  of the textbook by Yong and Zhou \cite{YongZhou}, we shall use:

\begin{theorem}
\label{thm:SMP}
Under assumptions (A.1--4),  
if the mapping $[0,T] \ni t \hookrightarrow \mu_{t}\in\cP_2(\RR^d)$ is measurable and bounded, and the cost functional $J$ is defined by
\begin{equation}
\label{eq:5:5:2}
J\bigl(\beta;\mu\bigr)
= {\mathbb E}
\biggl[ g(U_{T},\mu_{T}) + \int_{0}^T f(t,U_{t},\mu_{t},\beta_{t}) dt \biggr],
\end{equation}
for any progressively measurable process $\beta=(\beta_{t})_{0 \leq t \leq T}$ satisfying the admissibility condition \eqref{eq:4:4:1} where $ U=(U_{t})_{0 \leq t \leq T}$ is the corresponding controlled diffusion process 
$$
U_{t} = x_{0} + \int_{0}^t b(s,U_{s},\mu_{s},\beta_{s}) ds + \sigma W_{t}, \quad t \in [0,T],
$$
for $x_{0} \in \RR^d$, and if the forward-backward system
\begin{equation}
\label{eq:4:4:6}
\begin{split}
&dX_{t} = b\bigl(t,X_{t},\mu_{t},\hat{\alpha}(t,X_{t},\mu_{t},Y_{t})\bigr) dt + \sigma dW_{t},\qquad  X_{0}=x_{0}
\\
&dY_{t} = - \partial_{x} H(t,X_{t},\mu_{t},Y_{t},\hat{\alpha}(t,X_{t},\mu_{t},Y_{t}) \bigr) + Z_{t} dW_{t},\qquad Y_{T} = \partial_{x} g(X_{T},\mu_{T})
\end{split}
\end{equation}
has a solution $(X_{t},Y_{t},Z_{t})_{0 \leq t \leq T}$ such that 
\begin{equation}
\label{eq:4:4:3}
{\mathbb E} \biggl[ \sup_{0 \leq t \leq T} \bigl( \vert X_{t} \vert^2 + \vert Y_{t} \vert^2 \bigr) + \int_{0}^T \vert Z_{t} \vert^2 
dt \biggr] < + \infty,
\end{equation}
if we set $\hat{\alpha}_{t} = \hat{\alpha}(t,X_{t},\mu_{t},Y_{t})$, then for any $\beta=(\beta_{t})_{0 \leq t \leq T}$ satisfying \eqref{eq:4:4:1}, it holds
$$
J\bigl(\hat{\alpha};\mu\bigr) + \lambda \EE \int_{0}^T \vert \beta_{t} - \hat{\alpha}_{t} \vert^2 dt
\leq J\bigl(\beta;\mu\bigr).
$$
\end{theorem}

\begin{proof}  By Lemma \ref{le:minimizer}, $\hat{\alpha}=(\hat{\alpha}_{t})_{0 \leq t \leq T}$ satisfies 
\eqref{eq:4:4:1}, and the standard proof of the stochastic maximum principle, see for example Theorem 6.4.6 in Pham
\cite{Pham.book} gives
\begin{equation*}
\begin{split}
J\bigl(\beta;\mu\bigr) &\geq J\bigl(\hat{\alpha}; \mu\bigr)
 + \EE \int_{0}^T 
\bigl[ H(t,U_{t},\mu_{t},Y_{t},\beta_{t}) - H(t,X_{t},\mu_{t},Y_{t},\hat{\alpha}_{t}) 
\\
&\hspace{-12pt} - \langle U_{t} - X_{t},\partial_{x} H(t,X_{t},\mu_{t},Y_{t},\hat{\alpha}_{t}) \rangle
- \langle \beta_{t} - \hat{\alpha}_{t},\partial_{\alpha} H(t,X_{t},\mu_{t},Y_{t},\hat{\alpha}_{t}) \rangle
\bigr] dt.
\end{split}
\end{equation*}
 By linearity of $b$ and assumption (A.2) on $b$, the Hessian of $H$ satisfies \eqref{fo:lambdaconvexity}, so that the required convexity assumption is satisfied. The result easily follows.
 \end{proof}

\begin{remark}
\label{rem:SMP0}
As the proof shows, the result of Theorem \ref{thm:SMP} above still holds if the control $\beta=(\beta_t)_{0\le t\le T}$ is merely adapted to a larger filtration
as long as the Wiener process $ W=(W_t)_{0\le t\le T}$ remains a Brownian motion for this filtration. 
\end{remark}

\begin{remark}
\label{rem:SMP}
Theorem \ref{thm:SMP} has interesting consequences. First, it says that the optimal control, if it exists, must be unique. 
Second, it also implies that, given two solutions $( X, Y, Z)$ and $( X', Y', Z')$ to \eqref{eq:4:4:6}, $d\PP \otimes dt$ a.e.  it holds
$$
\hat{\alpha}(t,X_{t},\mu_{t},Y_{t}) = \hat{\alpha}(t,X_{t}',\mu_{t},Y_{t}'),
$$
so that $X$ and $X'$ coincide by the Lipschitz property of the coefficients of the forward equation.
As a consequence, $(Y,Z)$ and $(Y',Z')$ coincide as well. 
\end{remark}

It should be noticed that in some sense, the bound provided by Theorem \ref{thm:SMP} is sharp within the realm of convex models as shown for example by  the following slight variation on the same theme. We shall use this form repeatedly in the proof of our main result.

\begin{proposition}
\label{prop:SMP}
Under the same assumptions and notation as in Theorem \ref{thm:SMP} above, if we consider in addition 
another measurable and bounded mapping $[0,T] \ni t \hookrightarrow \mu_{t}'\in\cP_2(\RR^d)$
and  the controlled diffusion process  $U'=(U'_{t})_{0 \leq t \leq T}$ defined by
$$
U_{t}' = x_{0}' + \int_{0}^t b(s,U_{s}',\mu_{s}',\beta_{s}) ds + \sigma W_{t}, \quad t \in [0,T],
$$
for an initial condition $x_{0}' \in \RR^d$ possibly different from $x_{0}$,  then, 
\begin{equation}
\label{eq:9:4:1}
\begin{split}
& J\bigl(\hat{\alpha};\mu\bigr) + 
\langle x_{0}'-x_{0},Y_{0} \rangle  + 
\lambda \EE \int_{0}^T \vert \beta_{t} - \hat{\alpha}_{t} \vert^2 dt
\\
&\hspace{15pt}
\leq J\bigl(\bigl[\beta,\mu'\bigr]; \mu\bigr) + \EE \biggl[ \int_{0}^T 
\langle b_{0}(t,\mu_{t}')
- b_{0}(t,\mu_{t}),Y_{t}\rangle dt \biggr],
\end{split}
\end{equation}
where 
\begin{equation}
\label{eq:7:6:1}
J\bigl(\bigl[\beta,\mu' \bigr]; \mu\bigr) = 
{\mathbb E}
\biggl[ g(U_{T}',\mu_{T}) + \int_{0}^T f(t,U_{t}',\mu_{t},\beta_{t}) dt \biggr].
\end{equation}
The parameter $[\beta,\mu']$ in the cost $J([\beta,\mu']; \mu)$ indicates that the flow of measures in the drift of $U'$ is $(\mu_{t}')_{0 \leq t \leq T}$ whereas the flow of measures in the cost functions  is $(\mu_{t})_{0 \leq t \leq T}$. In fact, we should also indicate that the initial condition $x_{0}'$ might be different from $x_{0}$, but we prefer not to do so since there is no risk of confusion in the sequel. Also, when $x_{0}'=x_{0}$ and $\mu_{t}'=\mu_{t}$ for any $t \in [0,T]$, $J([\beta,\mu']; 
\mu)=J(\beta; \mu)$.
\end{proposition}
 
\begin{proof}
The idea is to go back to the original proof of the stochastic maximum principle and using It\^o's formula, expand 
$$
\biggl( \langle U_{t}'-X_{t},Y_{t} \rangle + \int_{0}^t \bigl[ 
f(s,U_{s}',\mu_{s},\beta_{s}) - f(s,X_{s},\mu_{s},\hat{\alpha}_{s})  \bigr] ds \biggr)_{0 \leq t \leq T}.
$$
Since the initial conditions $x_{0}$ and $x_{0}'$ are possibly different, we get the additional  term
$\langle x_{0}'-x_{0},Y_{0} \rangle$  in the left hand side of \eqref{eq:9:4:1}. Similarly, since the drift of $ U'$ is driven by 
$(\mu_{t}')_{0 \leq t \leq T}$, we get the additional difference of the drifts in order to account for the fact that the drifts are driven by the different flows of probability measures. 
\end{proof}

\section{\textbf{The Mean-Field FBSDE}}
\label{sec:3}
In order to solve the standard stochastic control problem \eqref{fo:mfgcontrolpb} using the 
Pontryagin maximum principle, we minimize the Hamiltonian $H$ 
with respect to the control variable $\alpha$, and inject the minimizer $\hat\alpha$ into the forward equation of the state as well as the adjoint backward equation. Since the minimizer $\hat\alpha$ depends upon both the forward state $X_t$ and the adjoint process $Y_t$, this creates a strong coupling
between the forward and backward equations leading to the FBSDE \eqref{eq:4:4:6}. 
The MFG matching condition (iii) of Subsection \ref{subsec:2:2} then reads: seek a family of probability distributions $(\mu_{t})_{0 \leq t \leq T}$ of order 2 such that the process $X$ solving the forward equation of
\eqref{eq:4:4:6} admits $(\mu_{t})_{0 \leq t \leq T}$ as flow of marginal distributions. 

In a nutshell, the probabilistic approach to the solution of the mean-field game problem
results in the solution of a FBSDE of the McKean-Vlasov type
\begin{equation}
\label{fo:mfFBSDE}
\begin{split}
&dX_{t} = b \bigl( t,X_{t},\PP_{X_{t}},\hat{\alpha}(t,X_{t},\PP_{X_{t}},Y_{t})\bigr) dt + \sigma dW_{t},
\\
&dY_{t} = - \partial_{x} H\bigl(t,X_{t},\PP_{X_{t}},Y_t,\hat{\alpha}(t,X_{t},\PP_{X_{t}},Y_{t}) \bigr) dt  + Z_{t}dW_{t},
\end{split}
\end{equation}
with the initial condition $X_{0}=x_{0}\in \RR^d$, and terminal condition $Y_{T} = \partial_{x} g(X_{T},\PP_{X_{T}})$.  
To the best of our knowledge, this type of FBSDE has not been considered in the existing literature. However, our experience with the classical theory of FBSDEs tells us that existence and uniqueness are expected to hold in short time when the coefficients driving \eqref{fo:mfFBSDE} are Lipschitz-continuous in the variables $x$, $\alpha$ and $\mu$ from standard contraction arguments. This strategy can also be followed in the McKean-Vlasov setting, taking advantage of the Lipschitz regularity of the coefficients upon the parameter $\mu$ for the 2--Wasserstein distance, exactly as in the theory of McKean-Vlasov (forward) SDEs. See Sznitman \cite{Sznitman}. However, the short time restriction is not really satisfactory for many reasons, and in particular for practical applications.
Throughout the paper, all the regularity properties with respect to $\mu$ are understood in the sense of the 
$2$--Wasserstein's distance $W_2$. Whenever $E$ is a separable Banach space, for any $p\ge 1$, $\mu,\mu' \in {\mathcal P}_{p}(E)$, the distance $W_{p}(\mu,\mu')$ is defined  by:
\begin{equation*}
\begin{split}
&W_p(\mu,\mu')
\\
&=\inf\left\{\left[\int_{E\times E} |x-y|_{E}^p \, \pi(dx,dy)\right]^{1/p};\;\pi\in\cP_p(E\times E) \mbox{ with marginals } \mu \mbox{ and } \mu'\right\}.
\end{split}
\end{equation*}
Below, we develop an alternative approach and prove existence of a solution over arbitrarily prescribed time duration $T$. The crux of the proof is to take advantage of the convexity of the coefficients. Indeed, in optimization theory, convexity often leads to compactness. Our objective is then to take advantage of this compactness in order to solve the matching problem (iii) in \eqref{fo:mfgcontrolpb}
 by applying Schauder's fixed point theorem in an appropriate space of finite measures on ${\mathcal C}([0,T];\RR^d)$. 

For the sake of convenience, we restate the general FBSDE \eqref{fo:mfFBSDE} of McKean-Vlasov type in the special set-up of the present paper. It reads:
\begin{equation}
\label{fo:actual_mfFBSDE}
\begin{split}
&dX_{t} = \bigl[ b_0( t,\PP_{X_{t}})  + b_1(t)X_t +b_2(t)\hat{\alpha}(t,X_{t},\PP_{X_{t}},Y_{t})\bigr] dt + \sigma dW_{t},
\\
&dY_{t} = - \bigl[b_1^{\dagger}(t)Y_t +\partial_{x} f\bigl(t,X_{t},\PP_{X_{t}},\hat{\alpha}(t,X_{t},\PP_{X_{t}},Y_{t}) \bigr)\bigr]dt  + Z_{t}dW_{t},
\end{split}
\end{equation}
where $a^\dagger$ denotes the transpose of the matrix $a$.
\subsection{Standing Assumptions and Main Result}  
In addition to (A.1--4), we shall rely on the following assumptions.

\vskip 2pt\noindent
\textbf{(A.5)} The functions $[0,T] \ni t \hookrightarrow f(t,0,\delta_{0},0)$, $[0,T] \ni t \hookrightarrow \partial_{x}f(t,0,\delta_{0},0)$  and $[0,T] \ni t \hookrightarrow \partial_{\alpha}f(t,0,\delta_{0},0)$ are bounded by $c_L$, and, for all $t \in [0,T]$, $x,x' \in \RR^d$, $\alpha,\alpha' \in \RR^k$ and $\mu,\mu' \in {\mathcal P}_{2}(\RR^d)$, it holds:
\begin{equation*}
\begin{split}
&\bigl\vert (f,g)(t,x',\mu',\alpha') - (f,g)(t,x,\mu,\alpha)\bigr\vert 
\\
&\hspace{15pt} 
\leq c_{L} \bigl[ 1 + \vert (x',\alpha') \vert + \vert (x,\alpha) \vert + M_{2}(\mu) + M_{2}(\mu') \bigr]
 \bigl[ \vert (x',\alpha') - (x,\alpha) \vert +
W_{2}(\mu',\mu) \bigr].
\end{split}
\end{equation*} 
Moreover, $b_{0}$, $b_{1}$ and $b_{2}$ in \eqref{fo:drift} are bounded by $c_L$ and $b_{0}$ satisfies for any 
$\mu,\mu' \in {\mathcal P}_{2}(\RR^d)$:
$
\vert b_{0}(t,\mu') - b_{0}(t,\mu) \vert \leq c_L W_{2}(\mu,\mu')$.

\vskip 2pt\noindent
\textbf{(A.6)} For all $t \in [0,T]$, $x \in \RR^d$ and $\mu \in {\mathcal P}_{2}(\RR^d)$,
$\vert \partial_{\alpha }f(t,x,\mu,0) \vert \leq c_L$.

\vskip 2pt\noindent
\textbf{(A.7)} For all $(t,x) \in [0,T] \times \RR^d$, 
$\langle x, \partial_{x} f(t,0,\delta_{x},0) \rangle \geq - c_{L} (1+ \vert x \vert)$, 
$\langle x, \partial_{x} g(0,\delta_{x}) \rangle \geq - c_{L} (1+ \vert x \vert)$.

\vskip 6pt
\begin{theorem}
\label{prop:27:11:1}
Under (A.1--7), the forward-backward system \eqref{fo:mfFBSDE} has a solution. 
Moreover, for any solution $(X_{t},Y_{t},Z_{t})_{0 \leq t \leq T}$ to \eqref{fo:mfFBSDE},
there exists a function $u : [0,T] \times \RR^d \hookrightarrow \RR^d$ 
(referred to as the  FBSDE value function), satisfying the growth and Lipschitz properties
\begin{equation}
\label{eq:28:3:14}
\forall t \in [0,T], \quad \forall x,x' \in \RR^d, \quad 
\left\{
\begin{array}{l}
\vert u(t,x \vert \leq c ( 1+ \vert x \vert),
\\
\vert u(t,x) - u(t,x') \vert \leq c \vert x - x' \vert,
\end{array}
\right.
\end{equation}
for some constant $c\geq 0$, and such that, $\PP$-a.s., for all 
$t \in [0,T]$,  $Y_{t}=u(t,X_{t})$. In particular, for any $\ell \geq 1$, 
${\mathbb E}[\sup_{0 \leq t \leq T} \vert X_{t}\vert^{\ell}] < + \infty$. 
\end{theorem}

(A.5) provides Lipschitz continuity while condition 
(A.6) controls the smoothness of the running cost $f$ with respect to $\alpha$ uniformly in the other variables. 
The most unusual assumption is certainly condition (A.7). We refer to it as a \emph{weak mean-reverting} condition as it looks like a standard mean-reverting condition for recurrent diffusion processes. Moreover, as shown by 
the proof of Theorem \ref{prop:27:11:1}, Its role is to control the expectation of the forward equation in \eqref{fo:mfFBSDE} and to establish an a priori bound for it. This is of crucial importance in order to make the compactness strategy effective. We use the terminology \emph{weak} as it is not expected to converge with time. 

\begin{remark}
An interesting example which we should keep in mind is the so-called linear-quadratic model in which $b_{0}$, $f$ and $g$ have the form:
\begin{equation*}
b_{0}(t,\mu) = b_{0}(t) \overline  \mu,
\  g(x,\mu) = \frac{1}{2} \bigl\vert q x + \bar{q}
\overline \mu  \bigr\vert^2,
\ 
f(t,x,\mu,\alpha) = \frac{1}{2} \bigl\vert m(t) x + \bar{m}(t) \overline \mu \bigr\vert^2 
+ \frac{1}{2} \vert n(t) \alpha \vert^2,
\end{equation*}
where $q$, $\bar{q}$, $m(t)$ and $\bar{m}(t)$ are elements of $\RR^{d \times d}$, 
$n(t)$ is an element of $\RR^{k \times k}$ and $\overline \mu$ stands for the mean of $\mu$. In this framework, (A.7) says that
$\bar{q}^\dagger q \geq 0$ and $\bar{m}(t)^\dagger m(t) \geq 0$ in the sense of quadratic forms.
In the one-dimensional case $d=1$, (A.7) says that $q\bar{q}$ and $m(t) \bar{m}(t)$ must be non-negative. 
As shown in \cite{CarmonaDelarueLachapelle}, this condition is not optimal for existence, as the conditions $q(q+\bar{q}) \geq 0$ and 
$m(t) (m(t)+\bar{m}(t)) \geq 0$ are sufficient to guarantee the solvability of \eqref{fo:mfFBSDE}. Obviously, the gap between these conditions is the price to pay for treating general systems within a single framework. 
\end{remark}

\subsection{Rigorous Definition of the Matching Problem} 
The proof of Theorem \ref{prop:27:11:1} is split into four main steps. The first one consists in making the statement of the matching problem (iii) in \eqref{fo:mfgcontrolpb} rigorous. To this end, we need the following
\begin{lemma}
\label{lem:unq:frozen}
Given $\mu \in {\mathcal P}_{2}({\mathcal C}([0,T];\RR^d))$ with marginal distributions 
 $(\mu_{t})_{0 \leq t \leq T}$, the FBSDE \eqref{eq:4:4:6} is uniquely solvable. If we denote its solution by $(X_{t}^{x_{0};\mu},Y_{t}^{x_{0};\mu},Z_{t}^{x_{0};\mu})_{0 \leq t \leq T}$, then there exist a constant $c>0$, only depending upon the parameters of (A.1--7), and a locally bounded measurable function $u^{\mu} : [0,T] \times \RR^d \hookrightarrow \RR^d$ such that
 $$
\forall x,x' \in \RR^d,\quad\vert u^{\mu}(t,x') - u^{\mu}(t,x) \vert \leq c \vert x' - x \vert,  
$$
and $\PP$-a.s., for all $t \in [0,T]$, $Y_{t}^{x_{0};\mu} = u^{\mu}(t,X_{t}^{x_{0};\mu})$. 
\end{lemma}

\begin{proof} We know that $\partial_{x} H$ reads
$
\partial_{x} H(t,x,\mu,y,\alpha) = b_{1}^{\dagger}(t) y + \partial_{x} f(t,x,\mu,\alpha)$,
so that, by Lemma \ref{le:minimizer}, the driver 
$
[0,T] \times \RR^d \times \RR^d \ni (t,x,y) \hookrightarrow \partial_{x} H(t,x,\mu_{t},\hat{\alpha}(t,x,\mu_{t},y))
$ 
of the backward equation in \eqref{eq:4:4:6} is Lipschitz continuous in the variables
$(x,y)$, uniformly in $t$. Therefore, by standard results in FBSDE theory, existence and uniqueness hold when $T$ is small enough. Equivalently, when $T$ is arbitrary, there exists $\delta>0$, depending on the Lipschitz constant of the coefficients in the variables $x$ and $y$ such that  
unique solvability holds on $[T-\delta,T]$, that is when the initial condition $x_{0}$ of the forward process is prescribed at some time 
$t_{0} \in [T-\delta,T]$. The solution is then denoted by $(X_{t}^{t_{0},x_{0}},Y_{t}^{t_{0},x_{0}},Z_{t}^{t_{0},x_{0}})_{t_{0} \leq t \leq T}$. Following Delarue \cite{Delarue02}, existence and uniqueness hold on the whole $[0,T]$, provided
\begin{equation}
\label{eq:9:4:2}
\forall x_{0},x_{0}' \in \RR^d, \quad \bigl\vert Y_{t_{0}}^{t_{0},x_{0}} - Y_{t_{0}}^{t_{0},x_{0}'} \bigr\vert^2 
\leq c \vert x_{0} - x_{0}' \vert^2,
\end{equation}
for some constant $c$ independent of $t_{0}$ and $\delta$. Notice that, by Blumenthal's Zero-One Law, the random variables $Y_{t_{0}}^{t_{0},x_{0}}$ and 
$Y_{t_{0}}^{t_{0},x_{0}'}$ are deterministic.
By \eqref{eq:9:4:1}, we have
\begin{equation}
\label{eq:9:4:3}
\hat{J}^{t_{0},x_{0}} + \langle x_{0}'-x_{0},Y_{t_{0}}^{t_{0},x_{0}} \rangle
+ \lambda \EE \int_{t_{0}}^T \vert 
\hat{\alpha}_t^{t_{0},x_{0}}
- \hat{\alpha}_t^{t_{0},x_{0}'} \vert^2 dt 
\leq \hat{J}^{t_{0},x_{0}'},
\end{equation}
where $\hat{J}^{t_{0},x_{0}} = J( (\hat{\alpha}_{t}^{t_{0},x_{0}})_{t_{0} \leq t \leq T} ; \mu)$
and $\hat{\alpha}_t^{t_{0},x_{0}}
=
\hat{\alpha}(t,X_{t}^{t_{0},x_{0}},\mu_{t},Y_{t}^{t_{0},x_{0}})$ (with similar definitions for 
$\hat{J}^{t_{0},x_{0}'}$ and $\hat{\alpha}_t^{t_{0},x_{0}'}$ by replacing $x_{0}$ by $x_{0}'$).
Exchanging the roles of $x_{0}$ and $x_{0}'$ and 
adding the resulting inequality with \eqref{eq:9:4:3}, we deduce that 
\begin{equation}
\label{eq:9:4:6}
2 \lambda \EE \int_{t_{0}}^T \vert 
\hat{\alpha}_t^{t_{0},x_{0}}
- \hat{\alpha}_t^{t_{0},x_{0}'} \vert^2 dt \leq 
\langle x_{0}'-x_{0},Y_{t_{0}}^{t_{0},x_{0}'} - Y_{t_{0}}^{t_{0},x_{0}} \rangle.
\end{equation}
Moreover, by standard SDE estimates first and then by standard BSDE estimates, there exists a constant $c$ (the value of which may vary from line to line), independent of $t_{0}$ and $\delta$, such that 
$$
{\mathbb E} \bigl[ \sup_{t_{0} \leq t \leq T} \vert X_{t}^{t_{0},x_{0}} - X_{t}^{t_{0},x_{0}'} \vert^2 \bigr]
+ {\mathbb E} \bigl[ \sup_{t_{0} \leq t \leq T} \vert Y_{t}^{t_{0},x_{0}} - Y_{t}^{t_{0},x_{0}'} \vert^2 \bigr]
\leq c \EE \int_{t_{0}}^T \vert 
\hat{\alpha}_t^{t_{0},x_{0}}
- \hat{\alpha}_t^{t_{0},x_{0}'} \vert^2 dt.
$$
Plugging \eqref{eq:9:4:6} into the above inequality completes the proof of \eqref{eq:9:4:2}. 

The function $u^{\mu}$ is then defined
as $u^{\mu} : [0,T] \times \RR^d \ni (t,x) \hookrightarrow Y_{t}^{t,x}$. The representation property of $Y$ in terms of $X$ directly follows from \cite{Delarue02}. Local boundedness of $u^{\mu}$ follows from the Lipschitz continuity in the variable $x$ together with the obvious inequality:
\vskip 0pt
$\sup_{0 \leq t \leq T} \vert u^{\mu}(t,0) \vert \leq 
\sup_{0 \leq t \leq T} \biggl[
\EE \bigl[ \vert u^{\mu}(t,X_{t}^{0,0}) - u^{\mu}(t,0) \vert \bigr] + \EE \bigl[ \vert Y_{t}^{0,0}\vert \bigr] \biggr] < + \infty$. \end{proof}

\vskip 2pt\noindent
We now set 
 
\begin{definition} 
\label{def:4:12:1}
To each $\mu \in {\mathcal P}_{2}({\mathcal C}([0,T];\RR^d))$ with marginal distributions 
 $(\mu_{t})_{0 \leq t \leq T}$,  
we  associate the measure $\PP_{X^{x_{0};\mu}}$ where $X^{x_{0};\mu}$ is the solution 
of \eqref{eq:4:4:6} with initial condition $x_{0}$. The resulting mapping 
${\mathcal P}_{2}\bigl( {\mathcal C}([0,T];\RR^d) \bigr) \ni \mu
\hookrightarrow \PP_{X^{x_{0};\mu}} \in {\mathcal P}_{2}\bigl( {\mathcal C}([0,T];\RR^d) \bigr)$
is denoted by $\Phi$ and we call solution of the matching problem (iii) in \eqref{fo:mfgcontrolpb} any fixed point $\mu$ of $\Phi$. For such 
a fixed point $\mu$, $X^{x_{0};\mu}$ satisfies \eqref{fo:mfFBSDE}.
\end{definition}

Definition \ref{def:4:12:1} captures the essence of the approach of Lasry and Lions who freeze the probability measure at the optimal value when optimizing the cost. This is not the case in the study of the control of McKean-Vlasov dynamics, as investigated in \cite{CarmonaDelarue3}: in this different setting, optimization is also performed with respect to the measure argument. See also \cite{CarmonaDelarueLachapelle} and \cite{Bensoussanetal} for the linear quadratic case. 

\subsection{Existence under Additional Boundedness Conditions}
We first prove existence under an extra boundedness assumption.
\begin{proposition}
\label{prop:17:3:1}
The system \eqref{fo:mfFBSDE} is solvable if, in addition to (A.1--7), we also assume that $\partial_{x} f$ and $\partial_{x} g$ are uniformly bounded, i.e. for some constant $c_{B}>0$
\begin{equation}
\label{eq:6:5:1}
\forall t \in [0,T], \ x \in \RR^d, \ \mu \in {\mathcal P}_{2}(\RR^d), \ \alpha \in \RR^k, \quad 
\vert \partial_{x} g(x,\mu) \vert, \ \vert \partial_{x} f(t,x,\mu,\alpha) \vert
\leq c_{B}.
\end{equation}
\end{proposition}
Notice that \eqref{eq:6:5:1} implies (A.7).

\begin{proof} 
We apply Schauder's fixed point theorem in the space 
${\mathcal M}_{1}({\mathcal C}([0,T];\RR^d))$ of finite signed measure 
$\nu$ of order $1$ on ${\mathcal C}([0,T];\RR^d)$ endowed with the Kantorovich-Rubinstein norm:
$$
\| \nu \|_{\rm KR} = \sup \biggl\{ \biggl\vert \int_{{\mathcal C}([0,T];\RR^d)} F(w) d\nu(w) \biggr\vert \ ; \ F \in {\rm Lip}_{1}\bigl({\mathcal C}([0,T];\RR^d)\bigr) \biggr\},
$$
for $\nu \in {\mathcal M}_{1}({\mathcal C}([0,T];\RR^d))$, which is known to coincide with the Wasserstein distance $W_1$ on ${\mathcal P}_{1}({\mathcal C}([0,T];\RR^d))$.
In what follows, we prove existence by proving that there exists a closed convex subset ${\mathcal E}
\subset {\mathcal P}_{2}({\mathcal C}([0,T];\RR^d)) \subset 
{\mathcal M}_{1}({\mathcal C}([0,T];\RR^d))$ which is stable for $\Phi$, with a relatively compact range, $\Phi$ being continuous on ${\mathcal E}$. 

\textit{First Step.} We first establish several a priori estimates for the solution of \eqref{eq:4:4:6}. The coefficients $\partial_{x} f$ and $\partial_{x} g$ being bounded, the terminal condition in \eqref{eq:4:4:6} is bounded and the growth of the driver is of  the form:
$$
\vert \partial_{x} H\bigl(t,x,\mu_{t},y,\hat{\alpha}(t,x,\mu_{t},y)\bigr) \vert \leq 
c_{B} + c_{L} \vert y \vert.
$$
By standard BSDE estimates relying on Gronwall's lemma, this implies that there exists a constant $c$, only depending upon 
$c_{B}$, $c_{L}$ and $T$, such that, for any $\mu \in {\mathcal P}_{2}({\mathcal C}([0,T];\RR^d))$,  
\begin{equation}
\label{eq:10:4:5}
\forall t \in [0,T], \quad \vert Y_{t}^{x_{0};\mu} \vert \leq c
\end{equation}
holds $\PP$-almost surely. By \eqref{eq:3:5:1} in the proof of Lemma \ref{le:minimizer} and by (A.6), we deduce that (the value of $c$ possibly varying from line to line) 
\begin{equation}
\label{eq:10:4:2}
\forall t \in [0,T], \quad \hat{\alpha}\bigl(t,X_{t}^{x_{0};\mu},\mu_{t},Y_{t}^{x_{0};\mu}\bigr)
\leq c. 
\end{equation}
Plugging this bound into the forward part of \eqref{eq:4:4:6}, standard $L^p$ estimates 
for SDEs imply that there exists a constant $c'$, only depending upon $c_{B}$, $c_{L}$ and $T$, such that  
\begin{equation}
\label{eq:4:4:10}
{\mathbb E} \bigl[ \sup_{0 \leq t \leq T}
\vert X_{t}^{x_{0};\mu} \vert^4 \bigr]
\leq c'.
\end{equation}
We consider the restriction of $\Phi$ to the subset  ${\mathcal E}$ of probability measures of order $4$ whose fourth moment is not greater than $c'$, i.e.
$${\mathcal E} = \bigl\{ \mu \in {\mathcal P}_{4}\bigl({\mathcal C} ([0,T],\RR^d)\bigr)   : 
M_{4,{\mathcal C} ([0,T],\RR^d)} (\mu) \leq c' \bigr\},$$
${\mathcal E}$ is convex and closed for the $1$-Wasserstein distance and $\Phi$ maps ${\mathcal E}$ into itself. 

\textit{Second Step.} The family of processes 
$((X_{t}^{x_{0};\mu})_{0 \leq t \leq T})_{\mu \in {\mathcal E}}$ is tight in ${\mathcal C}([0,T];\RR^d)$, as a consequence of 
\eqref{eq:10:4:2} and \eqref{eq:4:4:10}.
By \eqref{eq:4:4:10} again, $\Phi({\mathcal E})$ is actually relatively compact for the $1$-Wasserstein distance on 
${\mathcal C}([0,T];\RR^d)$. Indeed, tightness says that it is relatively compact for the topology of weak convergence of measures
and \eqref{eq:4:4:10} says that any weakly convergent sequence  $(\PP_{X^{x_{0};\mu_{n}}})_{n \geq 1}$, with 
$\mu_{n} \in {\mathcal E}$ for any $n \geq 1$, is convergent for the $1$-Wasserstein distance.

\textit{Third Step.} We finally check that $\Phi$ is continuous on ${\mathcal E}$. Given another measure $\mu' \in {\mathcal E}$, we deduce from \eqref{eq:9:4:1} in Proposition \ref{prop:SMP} that:
\begin{equation}
\label{eq:10:4:1}
J\bigl( \hat{\alpha}; \mu \bigr)
+  \lambda \EE \int_{0}^T \vert \hat{\alpha}_{t}'
-\hat{\alpha}_{t}
\vert^2 dt 
\leq
J \bigl( \bigl[ \hat{\alpha}',\mu' \bigr] ; \mu \bigr) + 
\EE \int_{0}^T \langle b_{0}(t,\mu_{t}') - b_{0}(t,\mu_{t}),Y_{t} \rangle dt,
\end{equation}
where
$
\hat{\alpha}_{t} = \hat{\alpha}(t,X_{t}^{x_{0};\mu},\mu_{t},Y_{t}^{x_{0};\mu})$, for
$t \in [0,T]$, with a similar definition for $\hat{\alpha}_{t}'$ by replacing $\mu$ by $\mu'$.
By optimality of $\hat{\alpha}'$ for the cost functional $J(\cdot;\mu')$, we claim:
\begin{equation*}
J \bigl( \bigl[\hat{\alpha}',\mu'\bigr] ; \mu \bigr) 
\leq J \bigl( \hat{\alpha} ; \mu'\bigr) + 
J \bigl( \bigl[\hat{\alpha}',\mu' \bigr]  ; \mu \bigr) - J \bigl( \hat{\alpha}' ; \mu' \bigr),
\end{equation*}
so that \eqref{eq:10:4:1} yields
\begin{equation}
\label{eq:3:5:2}
\begin{split}
\lambda \EE \int_{0}^T \vert \hat{\alpha}_{t}'
-\hat{\alpha}_{t}
\vert^2 dt 
&\leq J \bigl( \hat{\alpha} ; \mu'\bigr) - J \bigl( \hat{\alpha} ; \mu \bigr) 
+
J \bigl( \bigl[\hat{\alpha}',\mu' \bigr]  ; \mu \bigr) - J \bigl( \hat{\alpha}' ; \mu' \bigr)
\\
&\hspace{15pt}
+ 
\EE \int_{0}^T \langle b_{0}(t,\mu_{t}') - b_{0}(t,\mu_{t}),Y_{t} \rangle dt.
\end{split}
\end{equation}
We now compare $J( \hat{\alpha} ;\mu')$ with $J( \hat{\alpha} ; \mu)$
(and similarly $J(\hat{\alpha}' ; \mu')$ with $J( [\hat{\alpha}',\mu'] ; \mu)$). We notice that 
$J( \hat{\alpha} ;\mu)$ is the cost associated with the flow of measures $(\mu_{t})_{0 \leq t \leq T}$ and the diffusion process $X^{x_{0},\mu}$ whereas 
$J( \hat{\alpha} ;\mu)$ is the cost associated with the flow of measures $(\mu_{t}')_{0 \leq t \leq T}$ and the controlled diffusion process $U$ satisfying
$$
dU_t=\bigl[ b_{0}(t,\mu_t') + b_{1}(t) U_{t} + b_{2}(t) \hat{\alpha}_{t} \bigr] dt + \sigma dW_t, \quad t \in [0,T]; \quad U_{0}=x_{0}.
$$
By Gronwall's lemma, there exists a constant $c$ such that 
$$
\EE \bigl[ \sup_{0 \leq t \leq T} \vert X_{t}^{x_{0},\mu} - U_{t} \vert^2 \bigr] \leq c \int_{0}^T W_{2}^2(\mu_{t},\mu_{t}') dt. 
$$
Since $\mu$ and $\mu'$ are in ${\mathcal E}$, we deduce from 
(A.5), \eqref{eq:10:4:2} and \eqref{eq:4:4:10} that 
$$
J \bigl( \hat{\alpha} ; \mu' \bigr) - J \bigl( \hat{\alpha} ; \mu \bigr)
\leq c \biggl( \int_{0}^T W_{2}^2 (\mu_{t},\mu_{t}') dt \biggr)^{1/2},
$$
with a similar bound for $J([\hat{\alpha}',\mu'] ; \mu) - J ( \hat{\alpha}' ; \mu')$ (the argument is even simpler as the costs are driven by the same processes),
so that, from \eqref{eq:3:5:2} and \eqref{eq:10:4:5} again, together with Gronwall's lemma to go back to the controlled SDEs,
$$
\EE \int_{0}^T \vert \hat{\alpha}_{t}'
-\hat{\alpha}_{t}
\vert^2 dt +
\EE \bigl[ \sup_{0 \leq t \leq T} \vert X_{t}^{x_{0};\mu} -
X_{t}^{x_{0};\mu'} \vert^2 \bigr] \leq  c \biggl( \int_{0}^T W_{2}^2(\mu_{t},\mu_{t}') dt \biggr)^{1/2}. 
$$
As probability measures in ${\mathcal E}$ have bounded moments of order 4, Cauchy-Schwartz inequality yields
(keep in mind that $W_1(\Phi(\mu),\Phi(\mu'))
\leq \EE [ \sup_{0 \leq t \leq T} \vert X_{t}^{x_{0};\mu} -
X_{t}^{x_{0};\mu'} \vert]$):
\begin{equation*}
W_1(\Phi(\mu),\Phi(\mu'))
\leq  c \biggl( \int_{0}^T W_{2}^2(\mu_{t},\mu_{t}') dt \biggr)^{1/4}
\leq c\biggl( \int_{0}^T W_{1}^{1/2}(\mu_{t},\mu_{t}') dt \biggr)^{1/4},
\end{equation*}
which shows that $\Phi$ is continuous on ${\mathcal E}$ with respect to the $1$-Wasserstein distance $W_1$  on 
${\mathcal P}_{1}({\mathcal C}([0,T];\RR^d))$. 
\end{proof}

\subsection{Approximation Procedure}
Examples of functions $f$ and $g$ which are convex in $x$ and such that
$\partial_{x} f$ and $\partial_{x}g$ are bounded are rather limited in number and scope. For instance, boundedness of $\partial_{x} f$ and
$\partial_{x} g$ fails in the typical case when $f$ and $g$ are quadratic with respect to $x$.  In order to overcome this limitation, we propose to approximate the cost functions $f$ and $g$ by two sequences $(f^n)_{n \geq 1}$ and $(g^n)_{n \geq 1}$, referred to as approximated cost functions, satisfying (A.1--7) uniformly with respect to $n \geq 1$, and such that, for any $n \geq 1$, 
equation \eqref{fo:mfFBSDE}, with $(\partial_{x} f,\partial_{x} g)$ replaced by $(\partial_{x} f^n,\partial_{x} g^n)$, has a solution $(X^n,Y^n,Z^n)$. In this framework, Proposition \ref{prop:17:3:1} says that such approximated FBSDEs are indeed solvable
when $\partial_{x} f^n$ and $\partial_{x} g^n$ are bounded for any $n \geq 1$.
Our approximation procedure relies on the following:

\begin{lemma}
\label{lem:5:5:1}
If there exist two sequences $(f^n)_{n \geq 1}$ and $(g^n)_{n \geq 1}$ such that

$(i)$ there exist two parameters $c_{L}'$ and $\lambda' >0$ such that,  for any $n \geq 1$, $f^n$ and $g^n$ satisfy (A.1--7) with respect to $\lambda'$ and $c_{L}'$; 

$(ii)$ $f^n$ (resp. $g^n$)  converges towards $f$ (resp. $g$) uniformly on any bounded subset of 
$[0,T] \times \RR^d \times {\mathcal P}_{2}(\RR^d) \times \RR^k$ (resp. $\RR^d \times {\mathcal P}_{2}(\RR^d)$); 

$(iii)$ for any $n \geq 1$, equation \eqref{fo:mfFBSDE}, with $(\partial_{x} f,\partial_{x} g)$ replaced by $(\partial_{x} f^n,\partial_{x} g^n)$, has a solution which we denote by $(X^n,Y^n,Z^n)$.
\vskip 1pt\noindent
Then,  equation \eqref{fo:mfFBSDE} is solvable. 
\end{lemma}

\begin{proof} We establish tightness of the processes $(X^n)_{n \geq 1}$ in order to extract a convergent subsequence. For any $n \geq 1$, we consider the approximated Hamiltonian
\begin{equation*}
H^{n}(t,x,\mu,y,\alpha) = \langle b(t,x,\mu,\alpha), y \rangle + f^n(t,x,\mu,\alpha),
\end{equation*}
together with its minimizer $\hat{\alpha}^n(t,x,\mu,y) = \textrm{argmin}_{\alpha} H^{n}(t,x,\mu,y,\alpha)$. 
Setting $\hat{\alpha}^n_{t} = \hat{\alpha}^n(t,X_{t}^n,\PP_{X^n_{t}},Y_{t}^n)$ for any 
$t \in [0,T]$ and $n \geq 1$, our first step will be to prove that 
\begin{equation}
\label{eq:17:3:6}
\sup_{n \geq 1} \EE \biggl[ \int_{0}^T \vert \hat{\alpha}^n_{s} \vert^2 ds \biggr] < + \infty. 
\end{equation}
Since $X^n$ is the diffusion process controlled by $(\hat{\alpha}^n_{t} )_{0 \leq t \leq T}$, we use Theorem \ref{thm:SMP} to 
compare its behavior to the behavior of a \textit{reference controlled process} $U^n$ 
whose dynamics are driven by a specific control $\beta^n$. We shall consider two different versions for $U^n$
corresponding to the following choices for $\beta^n$:
\begin{equation}
\label{eq:17:3:10}
(i) \  \beta^n_{s} = \EE(\hat{\alpha}_{s}^n)\quad\mbox{for}\;0 \leq s \leq T; \quad  
(ii) \ \beta^n \equiv 0. 
\end{equation}
For each of these controls, we compare the cost to the optimal cost by using the version of the stochastic maximum principle 
which we proved earlier, and subsequently, 
derive useful information on the optimal control $(\hat{\alpha}_{s}^n)_{0 \leq s \leq T}$. 

\textit{First Step.}
We first consider $(i)$ in \eqref{eq:17:3:10}. In this case
\begin{equation}
\label{fo:unoft}
U_t^n = x_{0} + \int_{0}^t \bigl[ b_{0}(s,\PP_{X^n_{s}}) + b_{1}(s) U_{s}^n + b_{2}(s) \EE(\hat\alpha^n_{s}) \bigr] ds
+ \sigma W_{t}, \quad t \in [0,T].
\end{equation}
Notice that taking expectations on both sides of \eqref{fo:unoft} shows that $\EE(U^n_{s}) = {\mathbb E}(X^n_{s})$, for $0 \leq s \leq T$, and that 
$$
\bigl[U_t^n-\EE(U_t^n)\bigr] =  \int_{0}^t  b_{1}(s) \bigl[U_{s}^n  -\EE (U_{s}^n) \bigr]  ds+ \sigma W_{t}, \quad t \in [0,T],
$$
from which it easily follows that $\sup_{n \geq 1} \sup_{0 \leq s \leq T} \textrm{Var}(U^n_{s}) < + \infty$. 

By Theorem \ref{thm:SMP}, with 
$g^n(\cdot,\PP_{X^n_{T}})$ as terminal cost and $(f^n(t,\cdot,\PP_{X^n_{t}},\cdot))_{0 \leq t \leq T}$ as running cost, we get
\begin{equation}
\label{eq:19:3:1}
\begin{split}
&\EE \bigl[ g^n\bigl(X^n_{T},\PP_{X^n_{T}}\bigr) \bigr] + \EE \int_{0}^T 
\bigl[ \lambda' \vert \hat{\alpha}_{s}^n - \beta_{s}^n \vert^2 + 
f^n \bigl(s,X_{s}^n,\PP_{X_{s}^n},\hat{\alpha}_{s}^n \bigr) \bigr] 
ds
\\
&\hspace{15pt}
\leq \EE \biggl[ g^n\bigl(U^n_{T},\PP_{X^n_{T}}\bigr) +  \int_{0}^T f^n \bigl(s,U_{s}^n,\PP_{X_{s}^n},\beta_{s}^n 
\bigr) 
ds \biggr].
\end{split}
\end{equation}
Using the fact that $\beta^n_{s} = \EE(\hat{\alpha}_{s}^n)$, the convexity condition in (A.2,4) and Jensen's inequality, we obtain:
\begin{equation}
\label{eq:19:3:5}
\begin{split}
&g^n\bigl(\EE(X^n_{T}),\PP_{X^n_{T}}\bigr) + \int_{0}^T 
\bigl[ \lambda'  \textrm{Var} (  \hat{\alpha}_{s}^n ) +
f^n \bigl(s,\EE(X_{s}^n),\PP_{X_{s}^n},\EE(\hat{\alpha}_{s}^n) \bigr) \bigr] 
ds 
\\
&\hspace{15pt}
\leq \EE \biggl[ g^n\bigl(U^n_{T},\PP_{X^n_{T}}\bigr) +  \int_{0}^T f^n \bigl(s,U_{s}^n,\PP_{X_{s}^n},\EE(\hat{\alpha}_{s}^n) \bigr) 
ds \biggr].
\end{split}
\end{equation}
By (A.5), we deduce that there exists a constant $c$, depending only on $\lambda$, $c_{L}$, $x_{0}$ and $T$, such that (the actual value of $c$ possibly varying from line to line)
\begin{equation*}
\begin{split}
&\int_{0}^T \textrm{Var} (  \hat{\alpha}_{s}^n ) ds 
\leq  c\bigl( 1+ \EE \bigl[ \vert U^n_{T} \vert^2 \bigr]^{1/2} + \EE \bigl[ \vert X^n_{T} \vert^2 \bigr]^{1/2}\bigr)
\EE \bigl[ \vert U^n_{T} - {\mathbb E}(X_{T}^n) \vert^2 \bigr]^{1/2} 
\\
&\hspace{15pt} + c \int_{0}^T \bigl( 1+ \EE \bigl[ \vert U^n_{s} \vert^2 \bigr]^{1/2} + \EE \bigl[ \vert X^n_{s} \vert^2 \bigr]^{1/2} +  \EE \bigl[ \vert \hat{\alpha}^n_{s} \vert^2 \bigr]^{1/2}\bigr)
\EE \bigl[ \vert U^n_{s} - {\mathbb E}(X_{s}^n) \vert^2 \bigr]^{1/2} ds.
\end{split}
\end{equation*}
Since ${\mathbb E}(X^n_{t})={\mathbb E}(U^n_{t})$ for any $t \in [0,T]$, we deduce from the uniform boundedness of the variance of $(U^n_{s})_{0 \leq s \leq T}$ that
\begin{equation}
\label{eq:19:3:2}
\int_{0}^T \textrm{Var} (  \hat{\alpha}_{s}^n ) ds 
\leq  c\biggl[ 1+ \sup_{0 \leq s \leq T} \EE[\vert X^n_{s} \vert^2]^{1/2} + \biggl( \EE \int_{0}^T \vert \hat{\alpha}_{s}^n \vert^2
ds \biggr)^{1/2}\biggr].
\end{equation}
From this, the linearity of the dynamics of $X^n$ and Gronwall's inequality, we deduce:
\begin{equation}
\label{eq:19:3:3}
\sup_{0 \leq s \leq T} \textrm{Var}(X^n_{s}) \leq c\biggl[ 1+ 
\biggl( \EE \int_{0}^T \vert \hat{\alpha}_{s}^n \vert^2
ds \biggr)^{1/2}
\biggr], 
\end{equation}
since
\begin{equation}
\label{eq:6:5:2}
\sup_{0 \leq s \leq T} \EE \bigl[ \vert X^n_{s} \vert^2 \bigr] \leq c\biggl[ 1+ 
\EE \int_{0}^T \vert \hat{\alpha}_{s}^n \vert^2
ds \biggr].
\end{equation}
Bounds like \eqref{eq:19:3:3} allow us to control for any $0 \leq s \leq T$, the Wasserstein distance between the distribution of 
$X^n_{s}$ and the Dirac mass at the point $\EE(X^n_{s})$. 
\vspace{4pt}

\textit{Second Step.}
We now compare $X^n$ to the process controlled by the null control. So we consider case 
$(ii)$ in \eqref{eq:17:3:10}, and now 
$$
U_t^n = x_{0} + \int_{0}^t \bigl[ b_{0}(s,\PP_{X^n_{s}}) + b_{1}(s) U_{s}^n  \bigr] ds
+ \sigma W_{t}, \quad t \in [0,T].
$$
Since no confusion is possible, we still denote the solution by $U^n$ although it is different form the one in the first step.
 By the boundedness of $b_{0}$ in (A.5), it holds $\sup_{n \geq 1} {\mathbb E}[\sup_{0 \leq s \leq T} \vert U^n_{s} \vert^2] < + \infty$. Using Theorem \ref{thm:SMP} as before in the derivation of \eqref{eq:19:3:1} and \eqref{eq:19:3:5}, we get
\begin{equation*}
\label{eq:19:3:5:b}
\begin{split}
&g^n\bigl(\EE(X^n_{T}),\PP_{X^n_{T}}\bigr) + \int_{0}^T 
\bigl[ \lambda'  \EE ( \vert \hat{\alpha}_{s}^n \vert^2) +
f^n \bigl(s,\EE(X_{s}^n),\PP_{X_{s}^n},\EE(\hat{\alpha}_{s}^n) \bigr) \bigr] 
ds 
\\
&\hspace{15pt}
\leq \EE \biggl[ g^n\bigl(U^n_{T},\PP_{X^n_{T}}\bigr) +  \int_{0}^T f^n \bigl(s,U_{s}^n,\PP_{X_{s}^n},0 \bigr) 
ds \biggr].
\end{split}
\end{equation*}
By convexity of $f^n$ with respect to $\alpha$ (see (A.2)) together with (A.6), we have
\begin{equation*}
\label{eq:4:5:2}
\begin{split}
&g^n\bigl(\EE(X^n_{T}),\delta_{\EE(X^n_{T})}\bigr) + \int_{0}^T 
\bigl[ \lambda' \EE \bigl( \vert \hat{\alpha}_{s}^n \vert^2\bigr) +
f^n \bigl(s,\EE(X_{s}^n),\PP_{X_{s}^n},0\bigr) \bigr] 
ds 
\\
&\hspace{15pt}
\leq \EE \biggl[ g^n\bigl(U^n_{T},\PP_{X^n_{T}}\bigr) +  \int_{0}^T f^n \bigl(s,U_{s}^n,\PP_{X_{s}^n},0 \bigr) 
ds \biggr]
  + c {\mathbb E} \int_{0}^T \vert \hat{\alpha}_{s}^n \vert ds
,
\end{split}
\end{equation*}
for some constant $c$,  independent of $n$.
Using (A.5) again, we obtain:
\begin{equation*}
\begin{split}
&g^n\bigl(\EE(X^n_{T}),\delta_{\EE(X^n_{T})}\bigr) + \int_{0}^T 
\bigl[ \lambda' \EE \bigl( \vert \hat{\alpha}_{s}^n \vert^2\bigr) +
f^n \bigl(s,\EE(X_{s}^n),\delta_{\EE(X_{s}^n)},0\bigr) \bigr] 
ds 
\\
&\hspace{15pt}
\leq  g^n\bigl(0,\delta_{\EE(X^n_{T})}\bigr) +  \int_{0}^T f^n \bigl(s,0,\delta_{\EE(X^n_{s})},0\bigr) 
ds  + c {\mathbb E} \int_{0}^T \vert \hat{\alpha}_{s}^n \vert ds
\\
&\hspace{30pt}
 + 
c \bigl( 1+ \sup_{0 \leq s \leq T} \bigl[ \EE\bigl[ \vert X_{s}^n \vert^2 \bigr]^{1/2} \bigr] \bigr)
\bigl( 1 +
 \sup_{0 \leq s \leq T}\bigl[{\rm Var}(X^n_{s})\bigr]^{1/2} \bigr),
\end{split}
\end{equation*}
the value of $c$ possibly varying from line to line. From \eqref{eq:6:5:2}, Young's inequality yields
\begin{equation*}
\begin{split}
&g^n\bigl(\EE(X^n_{T}),\delta_{\EE(X^n_{T})}\bigr) + \int_{0}^T 
\bigl[ \frac{\lambda'}{2} \EE \bigl( \vert \hat{\alpha}_{s}^n \vert^2\bigr) +
f^n \bigl(s,\EE(X_{s}^n),\delta_{\EE(X_{s}^n)},0\bigr) \bigr] 
ds 
\\
&\hspace{15pt}
\leq  g^n\bigl(0,\delta_{\EE(X^n_{T})}\bigr) +  \int_{0}^T f^n \bigl(s,0,\delta_{\EE(X^n_{s})},0\bigr) 
ds  + c \bigl( 1 + \sup_{0 \leq s \leq T}\bigl[{\rm Var}(X^n_{s})\bigr] \bigr).
\end{split}
\end{equation*}
By \eqref{eq:19:3:3}, we obtain
\begin{equation*}
\begin{split}
&g^n\bigl(\EE(X^n_{T}),\delta_{\EE(X^n_{T})}\bigr) + \int_{0}^T 
\bigl[ \frac{\lambda'}{2} \EE \bigl( \vert \hat{\alpha}_{s}^n \vert^2\bigr) +
f^n \bigl(s,\EE(X_{s}^n),\delta_{\EE(X_{s}^n)},0\bigr) \bigr] 
ds 
\\
&\hspace{15pt}
\leq  g^n\bigl(0,\delta_{\EE(X^n_{T})}\bigr) +  \int_{0}^T f^n \bigl(s,0,\delta_{\EE(X^n_{s})},0\bigr) 
ds  + c \biggl( 1 + \biggl[  \int_{0}^T \EE \bigl( \vert \hat{\alpha}_{s}^n \vert^2 \bigr) ds \biggr]^{1/2}\biggr).
\end{split}
\end{equation*}
Young's inequality and the convexity in $x$ of $g^n$ and $f^n$ from  (A.2,4) give:
$$
\big\langle \EE(X^n_{T}) ,\partial_{x} g^n\bigl(0,\delta_{\EE(X^n_{T})}\bigr) \big\rangle + \int_{0}^T 
\bigl[ \frac{\lambda'}{4}  \EE \bigl( \vert \hat{\alpha}_{s}^n \vert^2\bigr) +
\big\langle \EE(X^n_{s}),
\partial_{x} f^n \bigl(s,0,\delta_{\EE(X^n_{s})},0 \bigr) \big\rangle \bigr] 
ds 
\leq c.
$$
By (A.7), we have $\EE \int_{0}^T \vert \hat{\alpha}_{s}^n \vert^2 ds \leq c \bigl( 1 + \sup_{0 \leq s \leq T} {\mathbb E}\bigl[ \vert X_{s}^n \vert^2 \bigr]^{1/2} \bigr)$, and the bound  \eqref{eq:17:3:6} now follows from \eqref{eq:6:5:2}, and as a consequence
\begin{equation}
\label{eq:19:3:10} 
\EE[\sup_{0 \leq s \leq T} \vert X^n_{s} \vert^2] \leq c. 
\end{equation}
Using \eqref{eq:17:3:6} and \eqref{eq:19:3:10}, it is plain to prove that the processes $({X}^n)_{n \geq 1}$ are tight.

\vspace{4pt}

\textit{Third Step.} Let $\mu$ be the limit of a convergent subsequence $(\PP_{X^{n_{p}}})_{p \geq 1}$. By \eqref{eq:19:3:10}, $M_{2,{\mathcal C}([0,T],\RR^d)}(\mu) < + \infty$.
Therefore, by Lemma \ref{lem:unq:frozen}, FBSDE \eqref{eq:4:4:6} has a unique solution  $(X_{t},Y_{t},Z_{t})_{0 \leq t \leq T}$. Moreover, there exists $u : [0,T] \times \RR^d \hookrightarrow \RR^d$, which is $c$-Lipschitz in the variable $x$ for the same constant $c$ as in the statement of the lemma, such that $Y_{t} = u(t,X_{t})$ for any $t \in [0,T]$. In particular, 
\begin{equation}
\label{eq:3:5:4}
\sup_{0 \leq t \leq T} \vert u(t,0) \vert \leq 
\sup_{0 \leq t \leq T} \biggl[
\EE \bigl[ \vert u(t,X_{t}) - u(t,0) \vert \bigr] + \EE \bigl[ \vert Y_{t}\vert \bigr] \biggr] < + \infty.
\end{equation}
We deduce that there exists a constant $c'$ such that $
\vert u(t,x) \vert \leq c' (1+ \vert x \vert)$, for $t \in [0,T]$ and $x \in \RR^d$.
By \eqref{eq:3:5:1} and (A.6), we deduce that (for a possibly new value of $c'$)
$
\vert \hat{\alpha}(t,x,\mu_{t},u(t,x))\vert \leq c' ( 1+ \vert x \vert). 
$
Plugging this bound into the forward SDE satisfied by $X$ in \eqref{eq:4:4:6}, we deduce that
\begin{equation}
\label{eq:3:5:5}
\forall \ell \geq 1, \quad {\mathbb E} \bigl[ \sup_{0 \leq t \leq T} \vert X_{t} \vert^{\ell} \bigr] < + \infty,
\end{equation}
and, thus,
\begin{equation}
\label{eq:3:5:5:b}
{\mathbb E} \int_{0}^T \vert \hat{\alpha}_{t} \vert^2 dt < + \infty,
\end{equation}
with $\hat{\alpha}_{t}=\hat{\alpha}(t,X_{t},\mu_{t},Y_{t})$, for $t \in [0,T]$. 
We can now apply the same argument to any $(X^n_{t})_{0 \leq t \leq T}$, for any $n \geq 1$. We claim
\begin{equation}
\label{eq:3:5:6}
\forall \ell \geq 1, \quad \sup_{n \geq 1}{\mathbb E} \bigl[ \sup_{0 \leq t \leq T} \vert X_{t}^n \vert^{\ell} \bigr] < + \infty.
\end{equation}
Indeed, the constant 
$c$ in the statement of Lemma \ref{lem:unq:frozen} does depend on $n$. Moreover, the second-order moments 
of $\sup_{0 \leq t \leq T} \vert X^n_{t} \vert$ are bounded, uniformly in $n \geq 1$ by 
\eqref{eq:19:3:10}. 
By (A.5), the driver in the backward component in \eqref{eq:4:4:6} is at most of linear growth in $(x,y,\alpha)$, so that by 
\eqref{eq:17:3:6} and standard $L^2$ estimates for BSDEs, the second-order moments 
of $\sup_{0 \leq t \leq T} \vert Y^n_{t} \vert$ are uniformly bounded as well. This shows \eqref{eq:3:5:6} by repeating the proof of \eqref{eq:3:5:5}.
By \eqref{eq:3:5:5} and \eqref{eq:3:5:6}, we get that $\sup_{0 \leq t \leq T}W_{2}(\mu^{n_{p}}_{t},\mu_{t}) \rightarrow 0$ as $n$ tends to $+ \infty$,
with $\mu^{n_{p}}= \PP_{X^{n_{p}}}$.

Repeating the proof of \eqref{eq:3:5:2}, we have
\begin{equation}
\label{eq:5:5:5}
\begin{split}
\lambda' \EE \int_{0}^T \vert \hat{\alpha}_{t}^n
-\hat{\alpha}_{t}
\vert^2 dt 
&\leq J^n \bigl( \hat{\alpha} ; \mu^n \bigr) - J \bigl( \hat{\alpha} ; \mu \bigr) 
+
J \bigl( \bigl[ \hat{\alpha}^n, \mu^n \bigr] ; \mu \bigr) - J^n\bigl( \hat{\alpha}^n ; \mu^n\bigr)
\\
&\hspace{15pt}
+ 
\EE \int_{0}^T \langle b_{0}(t,\mu_{t}^n) - b_{0}(t,\mu_{t}),Y_{t} \rangle dt,
\end{split}
\end{equation}
where $J( \cdot ; \mu)$ is given by \eqref{eq:5:5:2} and $J^n( \cdot ; \mu^n)$ is defined in a similar way, but with $(f,g)$ and $(\mu_{t})_{0 \leq t \leq T}$ replaced by 
$(f^n,g^n)$ and $(\mu_{t}^n)_{0 \leq t \leq T}$; $J( [ \hat{\alpha}^n,\mu^n] ; \mu)$ is defined as in \eqref{eq:7:6:1}. With these definitions at hand, we notice that
\begin{equation*}
\begin{split}
&J^n \bigl( \hat{\alpha} ; \mu^n \bigr) - J\bigl( \hat{\alpha} ; \mu\bigr)
\\
&= \EE \bigl[ g^n(U_{T}^n,\mu_{T}^n) - g(X_{T},\mu_{T}) \bigr] 
+ \EE \int_{0}^T \bigl[ f^n\bigl(t,U_{t}^n,\mu_{t}^n,\hat{\alpha}_{t}\bigr) -  f \bigl(t,X_{t},\mu_{t},\hat{\alpha}_{t}\bigr) \bigr] dt,
\end{split}
\end{equation*}
where $U^n$ is the controlled diffusion process:
$$
dU^n_{t} = \bigl[ b_{0}(t,\mu_{t}^n) + b_{1}(t) U_{t}^n + b_{2}(t) \hat{\alpha}_{t} \bigr] dt + \sigma dW_{t}, \quad t \in [0,T]; \quad U^n_{0} = x_{0}.
$$
By Gronwall's lemma and by convergence of $\mu^{n_{p}}$ towards $\mu$ for the $2$--Wasserstein distance, we claim
that $U^{n_{p}} \rightarrow X$ as $p \rightarrow + \infty$, 
for the norm $\EE[ \sup_{0 \leq s \leq T} \vert \cdot_{s} \vert^2]^{1/2}$.
Using on one hand the uniform convergence of $f^n$ and $g^n$ towards $f$ and $g$ on bounded subsets of their respective domains, and on the other hand the convergence of $\mu^{n_{p}}$ towards $\mu$ together with 
the bounds (\ref{eq:3:5:5}--\ref{eq:3:5:5:b}), we deduce that 
$J^{n_{p}}( \hat{\alpha} ; \mu^{n_{p}}) \rightarrow J( \hat{\alpha}; \mu)$ as $p \rightarrow + \infty$.
Similarly, using the bounds (\ref{eq:17:3:6}--\ref{eq:3:5:5}--\ref{eq:3:5:6}), the other differences in the right-hand side in \eqref{eq:5:5:5} tend to 0 along the subsequence $(n_{p})_{p \geq 1}$ so that 
$\hat{\alpha}^{n_{p}} \rightarrow
\hat{\alpha}$ as $p \rightarrow + \infty$ in $L^2([0,T] \times \Omega,dt \otimes d\PP)$. 
We deduce that $X$ is the limit of the sequence $(X^{n_{p}})_{p \geq 1}$ 
for the norm ${\mathbb E}[\sup_{0 \leq s \leq T} \vert \cdot_{s} \vert^2 ]^{1/2}$. Therefore, $\mu$ matches the law of $X$ exactly, proving that equation \eqref{fo:mfFBSDE} is solvable. 
\end{proof}

\subsection{Choice of the Approximating Sequence}
\label{subsec:approx}

In order to complete the proof of Theorem \ref{prop:27:11:1}, we must specify the choice of the approximating sequence in Lemma \ref{lem:5:5:1}. Actually, the choice is performed in two steps. 
We first consider the case when the cost functions $f$ and $g$ are strongly convex in the variables $x$:

\begin{lemma}
\label{lem:17:3:1}
Assume that, in addition to (A.1--7), there exists a constant $\gamma >0$ such that the functions $f$ and $g$ satisfy 
(compare with \eqref{fo:lambdaconvexity}):
\begin{equation}
\label{eq:5:5:20}
\begin{split}
&f(t,x',\mu,\alpha') - f(t,x,\mu,\alpha) 
\\
&\hspace{15pt}- \langle (x'-x,\alpha'-\alpha),\partial_{(x,\alpha)} f(t,x,\mu,\alpha) \rangle \geq \gamma \vert x' -x \vert^2 + \lambda 
\vert \alpha' - \alpha \vert^2, 
\\
&g(x',\mu) - g(x,\mu) - \langle x'-x,\partial_{x} g(x,\mu) \rangle \geq \gamma \vert x' -x \vert^2. 
\end{split}
\end{equation} 
Then, there exist two positive constants $\lambda'$ and $c'_{L}$, depending only upon $\lambda$, $c_L$ and $\gamma$,  and two sequences of functions $(f^n)_{n \geq 1}$ and $(g^n)_{n \geq 1}$ such that

$(i)$ for any $n \geq 1$, $f^n$ and $g^n$ satisfy (A.1--7) with respect to the parameters $\lambda'$ and $c_{L}'$ and 
$\partial_{x} f^n$ and $\partial_{x} g^n$ are bounded,

$(ii)$ for any bounded subsets of 
$[0,T] \times \RR^d \times {\mathcal P}_{2}(\RR^d) \times \RR^k$, there exists an integer $n_{0}$, such that, for any $n \geq n_{0}$, 
$f^n$  and $g^n$ coincide with $f$ and $g$ respectively. 
\end{lemma}

The proof of Lemma \ref{lem:17:3:1} is a pure technical exercise in convex analysis, and for this reason, we postpone its proof  to an appendix at the end of the paper.

\subsection{Proof of Theorem \ref{prop:27:11:1}}
\label{subsec:proof}

Equation \eqref{fo:mfFBSDE} is solvable when, in addition to (A.1--7), $f$ and $g$ satisfy the convexity condition \eqref{eq:5:5:20}. Indeed, by Lemma \ref{lem:17:3:1}, there exists an approximating sequence $(f^n,g^n)_{n \geq 1}$ satisfying $(i)$ and $(ii)$ in the statement of Lemma \ref{lem:5:5:1}, and also $(iii)$ by 
Proposition \ref{prop:17:3:1}. When $f$ and $g$ satisfy (A.1--7) only, the assumptions of Lemma \ref{lem:5:5:1} are satisfied with the following approximating sequence:
$$
f_{n}(t,x,\mu,\alpha) = f(t,x,\mu,\alpha) + \frac{1}{n} \vert x \vert^2; \quad 
g_{n}(x,\mu)=g(x,\mu) + \frac{1}{n} \vert x \vert^2, 
$$
for $(t,x,\mu,\alpha) \in [0,T] \times \RR^d \times {\mathcal P}(\RR^d) \times \RR^k$ and $n \geq 1$. 
Therefore,  \eqref{fo:mfFBSDE} is solvable under (A.1--7).
Moreover, given an arbitrary solution to \eqref{fo:mfFBSDE}, the existence of a function $u$, as in the statement of Theorem \ref{prop:27:11:1}, follows from 
Lemma \ref{lem:unq:frozen} and 
\eqref{eq:3:5:4}. Boundedness of the moments of the forward process is then proven as in 
\eqref{eq:3:5:5}. $\Box$

\section{\textbf{Propagation of Chaos and Approximate Nash Equilibriums}}
\label{se:apNash}
While the rationale for the mean-field strategy proposed by  Lasry-Lions is clear given the nature of Nash equilibriums (as opposed to other forms of optimization
suggesting the optimal control of stochastic dynamics of the McKean-Vlasov type as studied in \cite{CarmonaDelarue3}), it may not be obvious how the  
solution of the FBSDE introduced and solved in the previous sections provides approximate Nash equilibrium for large games.
In this section, we prove just that. The proof relies on the fact that the FBSDE value function is Lipschitz continuous, standard arguments in the propagation of chaos theory, and the following specific result due to Horowitz et al. (see for example Section 10 in \cite{RachevRuschendorf}) which we state as a lemma for future reference:

\begin{lemma}
\label{le:Horowitz}
Given $\mu\in\cP_{d+5}(\RR^d)$, there exists a constant $c$ depending only upon $d$ and $M_{d+5}(\mu)$ (see the notation 
\eqref{eq:20:6:1}), such that
$$
{\mathbb E} \bigl[ W_{2}^2(\bar{\mu}^N,\mu) \bigr] \leq C N^{-2/(d+4)},
$$
where $\bar{\mu}^N$ denotes the empirical measure of any sample of size $N$ from $\mu$.
\end{lemma}

Throughout this section,  assumptions (A.1--7) are in force. We let $(X_{t},Y_{t},Z_{t})_{0 \leq t \leq T}$ be a solution 
of \eqref{fo:mfFBSDE} and $u$ be the associated FBSDE value function. We denote by $(\mu_{t})_{0 \leq t \leq T}$ the flow of marginal probability measures $\mu_t=\PP_{X_{t}}$, for $0 \leq t \leq T$. We also denote by $J$ the optimal cost of the limiting Mean-Field problem
\begin{equation}
\label{eq:28:3:20}
J = \EE \biggl[ g(X_{T},\mu_{T}) + \int_{0}^T f\bigl(t,X_{t},\mu_{t},\hat{\alpha}(t,X_{t},\mu_{t},Y_{t})\bigr) dt \biggr],
\end{equation} 
where as before, $\hat\alpha$ is the minimizer function constructed in Lemma \ref{le:minimizer}.
For convenience, we fix a sequence  $((W_{t}^i)_{0 \leq t \leq T})_{i\ge 1}$ of independent $m$-dimensional Brownian motions, and for each integer $N$, we consider the solution $(X_{t}^1,\dots,X_{t}^N)_{0 \leq t \leq T}$  of the system of $N$ stochastic differential equations
\begin{equation}
\label{eq:28:3:12}
dX_{t}^i = b\bigl(t,X_{t}^i,\bar{\mu}^N_{t},\hat{\alpha}\bigl(t,X_{t}^i,\mu_{t},u(t,X_{t}^i)\bigr) \bigr) dt + \sigma dW_{t}^i, 
 \qquad \bar{\mu}^N_{t} = \frac{1}{N} \sum_{j=1}^N \delta_{X_{t}^j},
\end{equation}
with $t \in [0,T]$ and $X_{0}^i = x_{0}$. Equation  (\ref{eq:28:3:12})  is well posed since  $u$ satisfies the regularity property \eqref{eq:28:3:14} and the minimizer $\hat{\alpha}(t,x,\mu_{t},y)$ was proven, in Lemma \ref{le:minimizer}, to be Lipschitz continuous and at most of linear growth 
in the variables $x$ and $y$, uniformly in $t \in [0,T]$.
The processes $(X^i)_{1 \leq i \leq N}$ give the dynamics of the private states of the $N$ players in the stochastic differential game of interest when
the players use the strategies
\begin{equation}
\label{fo:alphaNi}
\balpha_{t}^{N,i} = \hat{\alpha}(t,X_{t}^i,\mu_{t},u(t,X_{t}^i)),\qquad 0\le t\le T,\;\; i\in\{1,\cdots,N\}.
\end{equation}
These strategies are in closed loop form. They are even \emph{distributed} since at each time $t \in [0,T]$, a player only needs to know the state of his own private state in order to compute the value of the control to apply at that time. By boundedness of $b_{0}$ 
and by \eqref{eq:3:5:1} and 
\eqref{eq:28:3:14}, it holds
\begin{equation}
\label{eq:9:6:1}
\sup_{N \geq 1} \max_{1 \leq i \leq N} \biggl[ {\mathbb E} \bigl[ \sup_{0 \leq t \leq T} \vert X_{t}^i \vert^2 \bigr] + \EE \int_{0}^T 
\vert \bar{\alpha}^{N,i}_{t} \vert^2 dt \biggr] < + \infty.
\end{equation}
\vskip 2pt
For the purpose of comparison, we introduce the notation we use when the players choose a generic set of strategies, say $((\beta_{t}^i)_{0 \leq t \leq T})_{1 \leq i \leq N}$. In this case, the dynamics of the private state $U^i$
of player $i\in\{1,\cdots, N\}$ are given by: 
\begin{equation}
\label{eq:28:3:10}
dU_{t}^i = b \bigl(t,U_{t}^i,\bar{\nu}_t^{N},\beta_t^i \bigr) dt + \sigma dW_{t}^i, \qquad 
\bar{\nu}^N_{t} = \frac{1}{N} \sum_{j=1}^N \delta_{U_{t}^j},
\end{equation}
with $t \in [0,T]$ and $U_{0}^i = x_{0}$, and where $((\beta_{t}^i)_{0 \leq t \leq T})_{1 \leq i \leq N}$ are $N$ square-integrable $\RR^k$-valued processes that are progressively measurable with respect to the filtration generated by $(W^1,\dots,W^N)$.
For each $1 \leq i \leq N$, we denote by
\begin{equation}
\label{eq:28:3:11}
\bar{J}^{N,i}(\beta^1,\dots,\beta^N) =  {\mathbb E}
\biggl[ g\bigl(U_{T}^i,\bar{\nu}_{T}^N
\bigr) + \int_{0}^T f(t,U_{t}^i,\bar{\nu}_{t}^N,\beta_{t}^i) dt \biggr],
\end{equation}
the cost to the $i$th player. Our goal is to construct approximate Nash equilibriums for the $N$-player game 
from a solution of \eqref{fo:mfFBSDE}. We follow the approach used by Bensoussan et al. \cite{Bensoussanetal} in the linear-quadratic case. See also \cite{Cardaliaguet}.

\begin{theorem}
\label{th:apNash}
Under assumptions (A.1--7), the strategies $(\balpha_{t}^{N,i})_{0\le t\le T,\;1 \leq i \leq N}$ defined in \eqref{fo:alphaNi} form an approximate Nash equilibrium of the $N$-player game (\ref{eq:28:3:10}--\ref{eq:28:3:11}).  More precisely, there exists a constant $c>0$ and a sequence of positive numbers $(\epsilon_N)_{N\ge 1}$ such that, for each $N \geq 1$,
\vskip 1pt
(i) $\epsilon_N\le c N^{-1/(d+4)}$ ;
\vskip 1pt
(ii) for any player $i\in\{1,\cdots,N\}$ and any progressively measurable strategy 
$\beta^i=(\beta^i_t)_{0\le t\le T}$, such that $\EE \int_{0}^T \vert \beta_{t}^i \vert^2 dt < + \infty$, one has
\begin{equation}
\label{fo:apNash}
\bar{J}^{N,i}(\balpha^{1,N},\dots,\balpha^{i-1,N},\beta^i,\balpha^{i+1,N},\dots,\balpha^{N,N}) \ge \bar{J}^{N,i}(\balpha^{1,N},\cdots,\balpha^{N,N})-\epsilon_N.
\end{equation}
\end{theorem}

\begin{proof}  By symmetry (invariance under permutation) of the coefficients of the private states dynamics and costs, we only need to prove \eqref{fo:apNash} for $i=1$.
Given a progressively measurable process  $\beta^1=(\beta^1_t)_{0\le t\le T}$ satisfying $\EE \int_{0}^T \vert \beta_{t}^1 \vert^2 dt < + \infty$, let us use the quantities defined in \eqref{eq:28:3:10} and \eqref{eq:28:3:11} with $\beta^i_t= \balpha_{t}^{N,i}$ for $i \in \{2,\cdots, N\}$ and $t \in [0,T]$. By 
boundedness of $b_{0}$, $b_1$ and $b_2$ and by Gronwall's inequality, we get:
\begin{equation}
\label{fo:U1est}
\EE\bigg[\sup_{0 \leq t \leq T} \vert U_{t}^1 \vert^2\bigg] \leq c \biggl( 1+ \EE \int_{0}^T \vert \beta^1_{t} \vert^2 dt\bigg) .
\end{equation}
Using the fact that the strategies $(\balpha_{t}^{N,i})_{0\le t\le T}$ satisfy the square integrability condition of admissibility, the same argument gives:
\begin{equation}
\label{eq:8:5:5}
\EE\bigg[\sup_{0 \leq t \leq T} \vert U_{s}^i \vert^2\bigg] \leq c,
\end{equation}
for $2 \leq i \leq N$,
which clearly implies after summation:
\begin{equation}
\label{fo:Uest}
\frac1N\sum_{j=1}^N \EE\bigg[ \sup_{0\le t\le T} \vert U_{t}^j \vert^2\bigg] \leq c \biggl( 1+ \frac1N\EE \int_{0}^T \vert \beta^1_{t} \vert^2 dt 
\biggr).
\end{equation}

\vskip 2pt
For the next step of the proof we introduce the system of decoupled independent and identically distributed
states
$$
d\bX_{t}^i = b\bigl(t,\bX_{t}^i,\mu_t,\hat{\alpha}(t,\bX_{t}^i,\mu_t,u(t,\bX^i_t)) \bigr) dt + \sigma dW_{t}^i, \quad 0 \leq t \leq T.
$$
Notice that the stochastic processes $\bX^i$  are independent copies of $X$ and, in particular, $\PP_{\bX_{t}^i} = \mu_{t}$ for any $t \in [0,T]$ and $i\in\{1,\cdots,N\}$. We shall use the notation: 
$$
\hat{\alpha}_{t}^i = \hat{\alpha}\bigl(t,\bX_{t}^i,\mu_t,u(t,\bX^i_t)\bigr), \quad t \in [0,T], \quad i \in \{1,\dots,N\}.
$$
Using the regularity of the FBSDE value function $u$ and the uniform boundedness of the family $(M_{d+5}(\mu_{t}))_{0 \leq t \leq T}$
derived in Theorem 
\ref{prop:27:11:1} together with  the estimate recalled in Lemma \ref{le:Horowitz}, we can follow  Sznitman's  proof \cite{Sznitman} (see also Theorem 1.3 of \cite{JourdainMeleardWoyczynski}) and get
\begin{equation}
\label{fo:JMW1}
\max_{1 \leq i \leq N} {\mathbb E} \bigl[  \sup_{0 \leq t \leq T}\vert X_{t}^i - \bar{X}_{t}^i \vert^2 \bigr] \leq c N^{-2/(d+4)},
\end{equation}
(recall that $(X^1,\dots,X^N)$ solves (\ref{eq:28:3:12})), and this implies:
\begin{equation}
\label{fo:JMW2}
\sup_{0 \leq t \leq T} {\mathbb E} \bigl[ W_{2}^2 (\bar{\mu}_{t}^N,\mu_{t}) \bigr] \leq c N^{-2/(d+4)}.
\end{equation}
Indeed, for each $t\in[0,T]$,
\begin{equation}
\label{fo:W2est}
W_2^2(\bar\mu^N_t,\mu_t) \le \frac2N\sum_{i=1}^N|X^i_t-\bX^i_t|^2+2W_2^2\bigg(\frac1N\sum_{i=1}^N\delta_{\bX^i_t},\mu_t\bigg),
\end{equation}
so that, taking expectations on both sides and using \eqref{fo:JMW1} and Lemma  \ref{le:Horowitz}, we get the desired estimate \eqref{fo:JMW2}.
Using the local-Lipschitz regularity of the coefficients $g$ and $f$ together with Cauchy-Schwarz inequality, we get, for each $i\in\{1,\cdots,N\}$, 
\begin{equation*}
\begin{split}
&\bigl|J-\bar{J}^{N,i}(\balpha^{N,1},\dots,\balpha^{N,N})\bigr|
\\
&\hspace{15pt}=\bigg|\EE\bigg[g(\bX_T^i,\mu_T)+\int_0^Tf\bigl(t,\bX^i_t,\mu_t,\hat\alpha_{t}^i\bigr)dt
-g(X^i_T,\bar\mu^N_T)-\int_0^Tf\bigl(t,X^i_t,\bar\mu^N_t,\bar{\alpha}_{t}^{N,i}\bigr)dt\bigg]
\bigg|
\\
&\hspace{15pt}\le c  \EE\bigg[ \biggl( 1+ \vert \bX_{T}^i \vert^2 + \vert X_{T}^i \vert^2  + \frac{1}{N}
\sum_{j=1}^N \vert X_{T}^j \vert^2 \biggr) \biggr]^{1/2}  \EE \bigl[ \vert \bX_T^i-X^i_T|^2  
+ W_2^2 (\mu_T,\bar\mu^N_T) \bigr]^{1/2}
\\
&\hspace{30pt} +
c \int_0^T  \biggl\{ \EE \biggl[\biggl( 1+ \vert \bX_{t}^i \vert^2 
+ \vert X_{t}^i \vert^2 
+ \vert \hat{\alpha}_{t}^i \vert^2
+ \vert \bar{\alpha}_{t}^{N,i} \vert^2  + \frac{1}{N}
\sum_{j=1}^N \vert X_{t}^j \vert^2 \biggr) \biggr]^{1/2}
\\
&\hspace{60pt} \times \EE \bigl[ |\bX_t^i-X^i_t|^2 + \vert \hat{\alpha}_{t}^i - \bar{\alpha}_{t}^{N,i} \vert^2 + 
W_2^2(\mu_t,\bar\mu^N_t) \bigr]^{1/2}  \biggr\} dt,
\end{split}
\end{equation*}
for some constant $c>0$ which can change from line to line.
By \eqref{eq:9:6:1}, we deduce 
\begin{equation*}
\begin{split}
\bigl|J-\bar{J}^{N,i}(\balpha^{N,1},\dots,\balpha^{N,N})\bigr|
&\le c \EE \bigl[ \vert \bX_T^i-X^i_T|^2  
+ W_2^2(\mu_T,\bar\mu^N_T) \bigr]^{1/2} 
\\
&\hspace{-12pt} + c
 \biggl( \int_0^T   \EE \bigl[ |\bX_t^i-X^i_t|^2 + \vert \hat{\alpha}_{t}^i - \bar{\alpha}_{t}^{N,i} \vert^2 + W_2^2(\mu_t,\bar\mu^N_t) \bigr] dt \biggr)^{1/2}.
\end{split}
\end{equation*}
Now, by the Lipschitz property of the minimizer $\hat\alpha$ proven in Lemma \ref{le:minimizer} and by 
the Lipschitz property of $u$ in \eqref{eq:28:3:14}, we notice that 
$$
\vert \hat{\alpha}_{t}^i - \bar{\alpha}_{t}^{N,i} \vert
= \bigl|\hat\alpha\bigl(t,\bX_t^i,\mu_t,u(t,\bX_t^i)\bigr)-\hat\alpha\bigl(t,X^i_t,\mu_t,u(t,X^i_t)\bigr)\bigr|
\leq c |\bX_t^i-X^i_t|.
$$
Using \eqref{fo:JMW1} and \eqref{fo:JMW2}, this proves that, for any $1\leq i \leq N$, 
\begin{equation}
\label{fo:apNash1}
\bar{J}^{N,i}(\balpha^{1,N},\dots,\balpha^{N,N}) = J + O( N^{-1/(d+4)}).
\end{equation}
This suggests that, in order to prove inequality \eqref{fo:apNash} for $i=1$, we could restrict ourselves to compare $\bar{J}^{N,1}(\beta^1,\balpha^{2,N},\dots,\balpha^{N,N}) $ to $J$. Using the argument which led to  \eqref{fo:U1est}, \eqref{eq:8:5:5} and \eqref{fo:Uest}, together with
the definitions of $U^j$ and $X^j$ for $j=1,\cdots,N$, we get, for any $t \in [0,T]$:
\begin{equation*}
\begin{split}
&\EE\bigg[\sup_{0 \leq s \leq t} \vert U_{t}^1 - X_{t}^1 \vert^2 \bigg]
\leq \frac{c}{N}\int_0^t \sum_{j=1}^N \EE \bigg[\sup_{0\le r\le s} \vert U_{r}^j  - X_{r}^j \vert^2\bigg] ds +
c \EE \int_{0}^T \vert \beta_{t}^1 - \balpha_{t}^{N,1} \vert^2 dt,
\\
&\EE\bigg[\sup_{0 \leq s \leq t} \vert U_{t}^i - X_{t}^i \vert^2 \bigg]
\leq \frac{c}{N}\int_0^t \sum_{j=1}^N \EE \bigg[\sup_{0\le r\le s} \vert U_{r}^j  - X_{r}^j \vert^2\bigg] ds, \quad 2 \leq i \leq N.
\end{split}
\end{equation*}
Therefore, using Gronwall's inequality, we get:
\begin{equation}
\label{fo:U-Xest}
\frac{1}{N} \sum_{j=1}^N \EE \bigg[\sup_{0\le t\le T}  \vert U_{t}^j  - X_{t}^j \vert^2\bigg]
\leq \frac{c}{N} \EE \int_{0}^T \vert \beta_{t}^1 - \balpha_{t}^{N,1} \vert^2 dt,
\end{equation}
so that 
\begin{equation}
\label{fo:Ui-Xiest}
\sup_{0 \leq t \leq T}{\mathbb E} \bigl[ \vert U_{t}^i - X_{t}^i \vert^2 \bigr]
\leq \frac{c}{N} \EE \int_{0}^T \vert \beta_{t}^{1} - \balpha_{t}^{N,1} \vert^2 dt, \quad  2 \leq i \leq N.
\end{equation}
Putting together \eqref{eq:9:6:1},
\eqref{fo:JMW1} and \eqref{fo:Ui-Xiest}, we see that, for any $A>0$, there exists a constant $c_{A}$ depending on $A$ such that 
\begin{equation}
\label{fo:JMW3}
\EE \int_{0}^T \vert \beta_{t}^1 \vert^2 dt \leq A\quad\Longrightarrow\quad \max_{2\le i\le N}\sup_{0 \leq t \leq T}{\mathbb E} \bigl[ \vert U_{t}^i - \bX_{t}^i \vert^2 \bigr]
\leq c_{A} N^{-2/(d+4)}.
\end{equation}
Let us fix $A>0$ (to be determined later) and assume that $\EE \int_{0}^T \vert \beta_{t}^1 \vert^2 dt \leq A$. Using 
\eqref{fo:JMW3} we see that 
\begin{equation}
\label{fo:JMW4}
\frac1{N-1}\sum_{j=2}^N
\EE\bigl[|U^j_t-\bX^j_t|^2\bigr]\leq c_{A} N^{-2/(d+4)},
\end{equation}
for a constant $c_A$ depending upon $A$, and whose value can change from line to line. Now by the triangle inequality for the Wasserstein distance:
\begin{equation}
\label{eq:8:5:1}
\begin{split}
&\EE\bigl[W_{2}^2(\bar{\nu}^N_{t},\mu_{t}) \bigr]
 \le c\bigg\{ \EE\bigg[W_{2}^2\bigg(\frac1{N}\sum_{j=1}^N\delta_{U^j_t},\frac1{N-1}\sum_{j=2}^N\delta_{U^j_t}\bigg) \bigg]
\\ 
&\hspace{30pt}+\frac1{N-1}\sum_{j=2}^N
\EE\bigl[|U^j_t-\bX^j_t|^2\bigr] +\EE\bigg[W_{2}^2\bigg(\frac1{N-1}\sum_{j=2}^N\delta_{\bX^j_t},\mu_{t}\bigg) \bigg] \biggr\}.
\end{split}
\end{equation}
Noticing that
\begin{equation*}
\EE\bigg[W_{2}^2\bigg(\frac1{N}\sum_{j=1}^N\delta_{U^j_t},\frac1{N-1}\sum_{j=2}^N\delta_{U^j_t}\bigg) \bigg]
\leq
\frac1{N(N-1)}\sum_{j=2}^N\EE\bigl[|U^1_t-U^j_t|^2\bigr],
\end{equation*}
which is $O(N^{-1})$ because of \eqref{fo:U1est} and \eqref{fo:Uest}. Plugging this inequality into
\eqref{eq:8:5:1}, and using \eqref{fo:JMW4} to control the second term and Lemma \ref{le:Horowitz} to estimate the third term therein, we conclude that
\begin{equation}
\label{fo:W2est4nu}
{\mathbb E} \bigl[ W_{2}^2(\bar{\nu}^N_{t},\mu_{t}) \bigr] \leq c_{A} N^{-2/(d+4)}.
\end{equation}

\vskip 2pt
For the final step of the proof we define $(\bar{U}^1_{t})_{0 \leq t \leq T}$ as the solution of the SDE
$$
d \bar{U}^1_{t} = b(t,\bar{U}_{t}^1,\mu_{t},\beta_{t}^1) dt + \sigma dW_{t}^1, \quad 0 \leq t \leq T, \; \bar{U}^1_0=x,
$$
so that, from the definition \eqref{eq:28:3:10} of $U^1$ we get:
$$
U^1_t- \bar{U}^1_{t} = \int_0^t[b_0(s,\mu_s)-b_0(s,\bar\nu^N_s)]ds+\int_0^tb_1(s)[U^1_s- \bar{U}^1_s ]ds.
$$
Using the Lipschitz property of $b_0$, \eqref{fo:W2est4nu}  and the boundedness of $b_1$ and applying Gronwall's inequality, we get
\begin{equation}
\label{fo:U-Ubarest}
\sup_{0 \leq t \leq T}{\mathbb E} \bigl[ \vert U_{t}^1  - \bar{U}_{t}^1 \vert^2 \bigr] \leq c_{A} N^{-2/(d+4)},
\end{equation}
so that, going over the computation leading to \eqref{fo:apNash1} once more and using \eqref{fo:W2est4nu},
\eqref{fo:U1est}, \eqref{eq:8:5:5}
and
 \eqref{fo:Uest}:
$$
\bar{J}^{N,1}(\beta^1,\balpha^{N,2},\dots,\balpha^{N,N}) \geq J(\beta^1) - c_{A} {N}^{-1/(d+4)},
$$
where $J(\beta^1)$ stands for the \emph{mean-field cost} of $\beta^1$:
\begin{equation}
J(\beta^1) = \EE \biggl[ g(\bar{U}_{T}^1,\mu_{T}) + \int_{0}^T f\bigl(t,\bar{U}_{t}^1,\mu_{t},\beta_{t}^1\bigr) dt \biggr].
\end{equation}
Since $J \leq J(\beta^1)$ (notice that, even though $\beta^1$ is adapted to a larger filtration than the filtration of $W^1$, the stochastic maximum principle still applies as pointed out in Remark \ref{rem:SMP0}), we get in the end
\begin{equation}
\label{fo:apNash2}
\bar{J}^{N,1}(\beta^1,\balpha^{N,2},\dots,\balpha^{N,N}) \geq J - c_{A} N^{-1/(d+4)},
\end{equation}
and from \eqref{fo:apNash1} and \eqref{fo:apNash2}, we easily derive the desired inequality \eqref{fo:apNash}. Actually, the combination of \eqref{fo:apNash1} and \eqref{fo:apNash2} shows that $(\balpha^{N,1},\dots,\balpha^{N,N})$ is an $\epsilon$-Nash equilibrium for $N$ large enough, with a precise quantification (though not optimal) of the relationship between $N$ and $\epsilon$. But for the proof to be complete in full generality, we need to explain how we choose $A$, and discuss what happens when $\EE\int_0^T|\beta^1_t|^2 dt>A$.

Using the convexity in $x$ of $g$ around $x=0$ and the convexity of $f$ in $(x,\alpha)$ around $x=0$ and $\alpha=0$, see 
\eqref{fo:lambdaconvexity}, we get:
\begin{equation*}
\begin{split}
&\bar{J}^{N,1}(\beta^1,\balpha^{N,2},\dots,\balpha^{N,N}) \\
&\hskip 15pt \geq  \EE \biggl[ g(0,\bar{\nu}^N_{T}) + \int_{0}^T f(t,0,\bar{\nu}^N_{t},0) dt \biggr]
+ \lambda {\mathbb E} \int_{0}^T \vert \beta_{t}^1 \vert^2 dt
\\
&\hskip 25pt + \EE \biggl[ \langle U_{T}^1 , \partial_{x} g(0,\bar{\nu}_{T}^N) \rangle
+ \int_{0}^T \bigl( \langle U_{t}^1, \partial_{x} f(t,0,\bar{\nu}_{t}^N,0) \rangle +
\langle \beta_{t}^1, \partial_{\alpha} f(t,0,\bar{\nu}_{t}^N,0) \rangle \bigr) dt \biggr].
\end{split}
\end{equation*}
The local-Lipschitz assumption with respect to the Wasserstein distance and the definition of the latter imply the existence of a constant $c>0$ such that for any $t \in [0,T]$, 
\begin{equation*}
\EE \bigl[ \vert f(t,0,\bar{\nu}^N_{t},0) - f(t,0,\delta_{0},0) \vert \bigr] \leq c \EE \bigl[ 1 + 
M_{2}^2(\bar{\nu}^N_{t}) \bigr]
= c
\biggl[ 1 + 
\biggl(\frac{1}{N}  \sum_{i=1}^N \EE \bigl[\vert U_{t}^i \vert^2 \bigr] \biggr) \biggr].
\end{equation*}
with a similar inequality for $g$.  From this, we deduce
\begin{equation*}
\begin{split}
&\bar{J}^{N,1}(\beta^1,\balpha^{N,2},\dots,\balpha^{N,N})  \geq   g(0,\delta_{0}) + \int_{0}^T f(t,0,\delta_{0},0) dt 
\\
&\hspace{30pt}+ \EE \biggl[ \langle U_{T}^1 , \partial_{x} g(0,\bar{\nu}_{T}^N) \rangle
+ \int_{0}^T \bigl( \langle U_{t}^1, \partial_{x} f(t,0,\bar{\nu}_{t}^N,0) \rangle 
+ \langle \beta_{t}^1, \partial_{\alpha} f(t,0,\bar{\nu}_{t}^N,0) \rangle\bigr) dt \biggr]
\\
&\hspace{30pt}+ \lambda {\mathbb E} \int_{0}^T \vert \beta_{t}^1 \vert^2 dt - c
\biggl[ 1 + 
\biggl(\frac{1}{N}  \sum_{i=1}^N \sup_{0 \leq t \leq T} \EE \bigl[\vert U_{t}^i \vert^2 \bigr] \biggr) \biggr].
\end{split}
\end{equation*}
By (A.5), we know that $\partial_{x} g$, $\partial_{x} f$ and $\partial_{\alpha} f$ are at most of linear growth in the measure parameter (for the $L^2$-norm), so that, for any $\delta >0$, there exists a constant $c_{\delta}$ such that 
\begin{equation}
\label{eq:29:2:102}
\begin{split}
\bar{J}^{N,1}(\beta^1,\balpha^{N,2},\dots,\balpha^{N,N}) 
&\geq  g(0,\delta_{0}) + \int_{0}^T f(t,0,\delta_{0},0) dt 
+ \frac{\lambda}{2} {\mathbb E} \int_{0}^T \vert \beta_{t}^1 \vert^2 dt
\\
&\hspace{-5pt} - \delta \sup_{0 \leq t \leq T} \EE \bigl[\vert U_{t}^1 \vert^2 \bigr]
- c_{\delta} \biggl(  1 + \frac{1}{N} \sum_{i=1}^N \sup_{0 \leq t \leq T} \EE \bigl[ \vert U_{t}^i \vert^2 \bigr] \biggr).
\end{split}
\end{equation}
Estimates \eqref{fo:U1est} and \eqref{eq:8:5:5} show that one can choose $\delta$ small enough in \eqref{eq:29:2:102} and $c$ so that 
$$
\bar{J}^{N,1}(\beta^1,\balpha^{N,2},\dots,\balpha^{N,N}) 
\geq 
- c+ \bigl( \frac{\lambda}{4} - \frac{c}{N} \bigr) 
{\mathbb E} \int_{0}^T \vert \beta_{t}^1 \vert^2 dt.
$$
This proves that there exists an integer $N_{0}$ such that, for any integer $N \geq N_{0}$ and constant $\bar A>0$, one can choose $A>0$ such that 
\begin{equation}
{\mathbb E} \int_{0}^T \vert \beta_{t}^1 \vert^2 dt \geq A
\quad\Longrightarrow\quad
\bar{J}^{N,1}(\beta^1,\balpha^{N,2},\dots,\balpha^{N,N})  \geq J+\bar A,
\end{equation}
which provides us with the appropriate tool to choose $A$ and avoid having to consider $(\beta^1_t)_{0\le t\le T}$ whose expected square integral is too large.
\end{proof}

\vskip 4pt
A simple inspection of the last part of the above proof shows that a stronger result actually holds when 
${\mathbb E} \int_{0}^T \vert \beta_{t}^1 \vert^2 dt \leq A$. Indeed, 
the estimates \eqref{fo:U1est}, \eqref{fo:JMW3}
 and \eqref{fo:W2est4nu}
can be used as in \eqref{fo:apNash1} to deduce (up to a modification of $c_{A}$)
\begin{equation}
\label{eq:28:3:30}
\bar{J}^{N,i}(\beta^1,\balpha^{N,2},\dots,\balpha^{N,N}) \geq J - c_{A} N^{-1/(d+4)}, \quad 2 \leq i \leq N.
\end{equation}

\begin{corollary}
Under assumptions (A.1--7), not only does 
$$\bigl(\balpha_{t}^{N,i} = \hat{\alpha}(t,X_{t}^i,\mu_{t},u(t,X_{t}^i)))_{1 \leq i \leq N}\bigr)_{0 \leq t \leq T}$$  form an approximate Nash equilibrium of the $N$-player game (\ref{eq:28:3:10}--\ref{eq:28:3:11}) but:

$(i)$ there exists an integer $N_{0}$ such that, for any $N \geq N_{0}$ and $\bar A>0$, there exists a constant $A>0$ such that, for any player $i\in\{1,\cdots,N\}$ and any admissible strategy $\beta^i=(\beta^i_t)_{0\le t\le T}$, 
 \begin{equation}
 \label{fo:aNash1}
\EE \int_{0}^T \vert \beta_{t}^i \vert^2 dt \geq A
\quad\Longrightarrow\quad \bar{J}^{N,i}(\balpha^{1,N},\dots,\balpha^{i-1,N},\beta^i,\balpha^{i+1,N},\dots,\balpha^{N,N})\geq J + \bar A.
\end{equation}

$(ii)$ Moreover, for any $A>0$, there exists a sequence of positive real numbers $(\epsilon_{N})_{N \geq 1}$ converging toward 
$0$, such that for any admissible strategy $\beta^1=(\beta^1_t)_{0\le t\le T}$ for the first player
\begin{equation}
 \label{fo:aNash2}
 \EE \int_{0}^T \vert \beta_{t}^1 \vert^2 dt \leq A
\quad\Longrightarrow\quad
\min_{1\le i\le N}\bar{J}^{N,i}(\beta^1,\balpha^{2,N},\dots,\balpha^{N,N}) \geq J - \varepsilon_{N}.
\end{equation}
\end{corollary}

\section{\textbf{Appendix: Proof of Lemma \ref{lem:17:3:1}}}
We focus on the approximation of the running cost $f$ (the case of the terminal cost $g$ is similar) and we ignore the dependence of $f$ upon $t$ to simplify the notation. For any $n \geq 1$, we define $f_{n}$ as the truncated Legendre transform:
\begin{equation}
\label{eq:25:11:1}
f_{n}(x,\mu,\alpha) 
= \sup_{\vert y \vert \leq n} 
 \inf_{z \in \RR^d}
\bigl[ \langle y,x-z \rangle + f(z,\mu,\alpha)
\bigr],
\end{equation}
for $(x,\alpha) \in \RR^d \times \RR^k$ and $\mu \in {\mathcal P}_{2}(\RR^d)$. 
By standard properties of the Legendre transform of convex functions,
\begin{equation}
\label{eq:5:5:8}
f_{n}(x,\mu,\alpha) 
\leq \sup_{y \in \RR^d} 
 \inf_{z \in \RR^d}
\bigl[ \langle y,x-z \rangle + f(z,\mu,\alpha)
\bigr] = f(x,\mu,\alpha).
\end{equation}
Moreover, by strict convexity of $f$ in $x$, 
\begin{equation}
\label{eq:5:5:6}
\begin{split}
f_{n}(x,\mu,\alpha) &\geq \inf_{z \in \RR^d}
\bigl[ f(z,\mu,\alpha)
\bigr] 
\geq \inf_{z \in \RR^d}
\bigl[ \gamma \vert z \vert^2 + \langle \partial_{x} f(0,\mu,\alpha),z \rangle
\bigr] + f(0,\mu,\alpha)
\\
&\geq - \frac{1}{4 \gamma} \vert \partial_{x} f(0,\mu,\alpha) \vert^2 + f(0,\mu,\alpha),
\end{split}
\end{equation}
so that $f_{n}$ has finite real values. Clearly, it is also $n$-Lipschitz continuous in $x$. 
\vspace{4pt}

\textit{First Step.} 
We first check that the sequence $(f_{n})_{n \geq 1}$ converges towards $f$, uniformly on bounded subsets 
of $\RR^d \times {\mathcal P}_{2}(\RR^d) \times \RR^k$. 
So for any given $R>0$, we restrict ourselves to $\vert x \vert \leq R$ and $\vert \alpha \vert \leq R$, and
$\mu \in  {\mathcal P}_{2}(\RR^d)$, such that $M_{2}(\mu) \leq R$. 
By (A.5), there exists a constant $c>0$, independent of $R$, such that 
\begin{equation}
\label{eq:24:3:1}
\sup_{z \in \RR^d} 
\bigl[ \langle y,z \rangle  - f(z,\mu,\alpha)
\bigr] 
\geq \sup_{z \in \RR^d} 
\bigl[ \langle y,z \rangle 
 - c \vert z \vert^2 \bigr] - c ( 1 + R^2)
= \frac{\vert y \vert^2}{4 c} - c ( 1 +R^2).
\end{equation}
Therefore,
\begin{equation}
\label{eq:24:3:2}
\inf_{z \in \RR^d}
\bigl[\langle y,x-z \rangle  + f(z,\mu,\alpha)
\bigr] 
\leq R \vert y \vert -  \frac{\vert y \vert^2}{4c} + c ( 1 + R^2).
\end{equation}
By \eqref{eq:5:5:6} and (A.5), $f_{n}(t,x,\mu,\alpha) \geq - c(1+R^2)$, $c$ depending possibly on $\gamma$, 
so that optimization in the variable $y$ can be done over points $y^{\star}$ satisfying
\begin{equation}
\label{eq:5:5:7}
- c (1 + R^2) \leq R \vert y^{\star} \vert -  \frac{\vert y^{\star} \vert^2}{4c} + c ( 1 + R^2), 
\quad \textrm{that is} \quad
\vert y^{\star} \vert \leq c ( 1 + R),
\end{equation}
In particular, for $n$ large enough (depending on $R$), 
$$
f_{n}(x,\mu,\alpha) = \sup_{y \in \RR^d}  \inf_{z \in \RR^d}
\bigl[ \langle y,x-z \rangle   + f(z,\mu,\alpha)
\bigr]
= f(x,\mu,\alpha).
$$
So on bounded subsets of $\RR^d \times {\mathcal P}_{2}(\RR^d) \times \RR^k$,  
$f_{n}$ and $f$ coincide for $n$ large enough. In particular, for $n$ large enough, $f_{n}(0,\delta_{0},0)$, $\partial_{x}
f_{n}(0,\delta_{0},0)$ and $\partial_{\alpha} f_{n}(0,\delta_{0},0)$ exist, coincide with 
$f(0,\delta_{0},0)$, $\partial_{x}
f(0,\delta_{0},0)$ and $\partial_{\alpha} f(0,\delta_{0},0)$ respectively, and are bounded by 
$c_{L}$ as in (A.5). 
Moreover, still for $\vert x \vert \leq R$, $\vert \alpha \vert \leq R$ and $M_{2}(\mu) \leq R$, 
we see from \eqref{eq:5:5:8} and \eqref{eq:5:5:7} that optimization in $z$ can be reduced to 
$z^{\star}$ satisfying
$$
\langle y^{\star},x-z^{\star} \rangle + f(z^{\star},\mu,\alpha) \leq f(x,\mu,\alpha) \leq c(1+R^2),
$$
the second inequality following from (A.5). 
By strict convexity of $f$ in $x$, we obtain
$$
- c (1+R) \vert z^{\star} \vert + 
\gamma \vert z^{\star} \vert^2 + \langle \partial_{x} f(0,\mu,\alpha),z^{\star} \rangle + f(0,\mu,\alpha) \leq  c( 1+ R^2),
$$
so that, by (A.5), $\gamma \vert z^{\star} \vert^2 - c (1+R) \vert z^{\star} \vert 
\leq  c( 1+ R^2)$,
that is 
\begin{equation}
\label{eq:5:5:10}
\vert z^{\star} \vert \leq c (1 +R).
\end{equation}

\textit{Second Step.}
We now investigate the convexity property of $f_{n}(\cdot,\mu,\cdot)$, for  a given $\mu \in {\mathcal P}_{2}(\RR^d)$. For any $h \in \RR$, 
$x,e,y,z_{1},z_{2} \in \RR^d$ and $\alpha,\beta \in \RR^k$, with 
$\vert y \vert \leq n$ and
$\vert e \vert, \vert \beta \vert \leq 1$,  we deduce from the convexity of $f(\cdot,\mu,\cdot)$:
\begin{equation*}
\begin{split}
&2\inf_{z \in \RR^d} 
\bigl[ \langle y,x- z \rangle  + f(z,\mu,\alpha)
\bigr]
\\
&\leq  \bigg\langle y,(x+he -z_{1})+(x-he - z_{2}) \bigg\rangle  
 + 2 f\biggl(\frac{z_{1}+z_{2}}{2},\mu,\frac{(\alpha+h \beta)+(\alpha-h \beta)}{2} \biggr)
\\
&\leq   \langle y,x + h e -z_{1} \rangle  + f(z_{1},\mu,\alpha + h \beta) 
+   \langle y,x - h e -z_{2} \rangle  
 + f(z_{2},\mu,\alpha - h \beta)  - 2\lambda h^2. 
\end{split}
\end{equation*}
Taking infimum with respect to $z_{1},z_{2}$ and supremum with respect to $y$, 
we obtain
\begin{equation}
\label{eq:24:3:3}
f_{n}(x,\mu,\alpha )
 \leq \frac{1}{2} f_{n}(x+he,\mu,\alpha + h \beta) + 
\frac{1}{2} f_{n}(x- h e,\mu,\alpha - h \beta) - \lambda  h^2. 
\end{equation}
In particular, the function $\RR^d \times \RR^k \ni (x,\alpha) \hookrightarrow 
f_{n}(x,\mu,\alpha) - \lambda \vert \alpha \vert^2$ is convex. We prove later on that it is also continuously 
differentiable so that \eqref{fo:lambdaconvexity}
 holds.
\vskip 4pt

In a similar way, we can investigate the semi-concavity property of $f_{n}(\cdot,\mu,\cdot)$. For any $h \in \RR$, $x,e,y_{1},y_{2} \in \RR^d$, $\alpha,\beta \in \RR^k$, with $\vert y_{1}\vert,\vert y_{2} \vert \leq n$ and $\vert e \vert, \vert \beta \vert  \leq 1$, 
\begin{equation*}
\begin{split}
&\inf_{z \in \RR^d}
\bigl[ \langle y_{1},x+h e-z \rangle  + f(z,\mu,\alpha+h \beta)
\bigr]  +  \inf_{z \in \RR^d}
\bigl[ \langle y_{2},x-h e-z \rangle  + f(z,\mu,\alpha - h \beta)
\bigr]
\\
&= \inf_{z \in \RR^d}
\bigl[ \langle y_{1},x-z \rangle  
 + f(z+h e,\mu,\alpha+h \beta)
\bigr] +  \inf_{z \in \RR^d}
\bigl[ \langle y_{2},x-z \rangle  + f(z - he,\mu,\alpha - h \beta)
\bigr].
\end{split}
\end{equation*}
By expanding $f(\cdot,\mu,\cdot)$ up to the second order, we see that
\begin{equation*}
\begin{split}
&\inf_{z \in \RR^d}
\bigl[ \langle y_{1},x+h e-z \rangle   + f(z,\mu,\alpha+h \beta)
\bigr]
 + \inf_{z \in \RR^d}
\bigl[ \langle y_{2},x- h e-z \rangle  + f(z,\mu,\alpha - h \beta)
\bigr]
\\
&\leq 
\inf_{z \in \RR^d}
\bigl[ \langle y_{1}+y_{2},x-z \rangle + 2 f(z,\mu,\alpha) \bigr]
+ c \vert h \vert^2,
\end{split}
\end{equation*}
for some constant $c$.  Taking the supremum over $y_{1},y_{2}$, we deduce that 
$$
f_{n}(x+he,\mu,\alpha+h \beta) + 
f_{n}(x-he,\mu,\alpha-h \beta)
- 2 f_{n}(x,\mu,\alpha)
\leq  c \vert h \vert^2.
$$
So for any $\mu \in {\mathcal P}_{2}(\RR^d)$, the function 
$\RR^d \times \RR^k \ni (x,\alpha)
\hookrightarrow f_{n}(x,\mu,\alpha) - c [ \vert x \vert^2 + \vert \alpha \vert^2]$
is concave and $f_{n}(\cdot,\mu,\cdot)$ 
is ${\mathcal C}^{1,1}$, the Lipschitz constant of the derivatives being uniform in $n \geq 1$ and $\mu \in 
{\mathcal P}_{2}(\RR^d)$. Moreover, by definition, the function $f_{n}(\cdot,\mu,\cdot)$ is $n$-Lipschitz continuous in the variable $x$, that is $\partial_{x} f_{n}$ is bounded, as required. 
\vskip 4pt

\textit{Third Step.} We now investigate (A.5). Given $\delta >0$, $R>0$ and $n \geq 1$, we consider $x \in \RR^d$, $\alpha \in \RR^k$, $\mu,\mu' \in {\mathcal P}_{2}(\RR^d)$ such that 
\begin{equation}
\label{eq:13:4:5}
\max\bigl(\vert x \vert,\vert \alpha \vert,M_{2}(\mu),M_{2}(\mu') \bigr) \leq R, \ \ W_{2}(\mu,\mu') \leq \delta.
\end{equation}
By (A.5) and \eqref{eq:5:5:10}, we can find a constant $c'$ (possibly depending on $\gamma$) such that
\begin{equation}
\label{eq:5:5:15}
\begin{split}
f_{n}(x,\mu',\alpha) &= \sup_{\vert y \vert \leq n}
\inf_{\vert z \vert \leq c (1+R)}
\bigl[ \langle y,x-z \rangle + f(z,\mu',\alpha) \bigr]
\\
&\leq 
\sup_{\vert y \vert \leq n}
\inf_{z \leq c (1+R)}
\bigl[ \langle y,x-z \rangle + f(z,\mu,\alpha) + c_{L}(1+R+\vert z \vert) \delta  \bigr]
\\
&= \sup_{\vert y \vert \leq n}
\inf_{z \in \RR^d}
\bigl[ \langle y,x-z \rangle + f(z,\mu,\alpha)  \bigr]
+ c'(1+R) \delta.
\end{split}
\end{equation}
This proves local Lipschitz-continuity in the measure argument as in (A.5). 

In order to prove local Lipschitz-continuity in the variables $x$ and $\alpha$, we use the ${\mathcal C}^{1,1}$-property. Indeed, for $x$, $\mu$ and 
$\alpha$ as in \eqref{eq:13:4:5}, we know that
\begin{equation}
\label{eq:5:5:12}
\bigl\vert \partial_{x} f_n(x,\mu,\alpha) \bigr\vert + 
\bigl\vert \partial_{\alpha} f_n(x,\mu,\alpha) \bigr\vert
\leq \bigl\vert \partial_{x} f_n(0,\mu,0) \bigr\vert + 
\bigl\vert \partial_{\alpha} f_n(0,\mu,0) \bigr\vert + cR.
\end{equation}
By \eqref{eq:5:5:8}, for any integer $p \geq 1$, there exists an integer $n_{p}$, such that, for any $n \geq n_{p}$, 
$f_{n}(0,\mu,0)$ and $f(0,\mu,0)$ coincide for  
$M_2(\mu) \leq p$. In particular, for $n \geq n_{p}$, 
\begin{equation}
\label{eq:5:5:16}
\bigl\vert \partial_{x} f_n(0,\mu,0) \bigr\vert + 
\bigl\vert \partial_{\alpha} f_n(0,\mu,\alpha) \bigr\vert
\leq  c \bigl(1+ M_{2}(\mu) \bigr) \quad \textrm{whenever} \quad M_{2}(\mu) \leq p,
\end{equation}
so that \eqref{eq:5:5:12} implies (A.5) whenever $n \geq n_{p}$ and $M_{2}(\mu) \leq p$. 
We get rid of these restrictions by modifying the definition of $f_{n}$. Given a probability measure $\mu \in {\mathcal P}_{2}(\RR^d)$ and an integer $p \geq 1$, we define $\Phi_{p}(\mu)$ as the push-forward of $\mu$ by the mapping
$\RR^d \ni x \hookrightarrow \bigl[\max\bigl(M_{2}(\mu),p \bigr) \bigr]^{-1} p x$
so that $\Phi_{p}(\mu) \in {\mathcal P}_{2}(\RR^d)$ and 
$
M_{2}( \Phi_{p}(\mu)) \leq \min( p,M_{2}(\mu))$. 
Indeed, if $X$ has $\mu$ as distribution, then the r.v.
$
X_{p} = p X/\max(M_{2}(\mu),p)$
has $\Phi_{p}(\mu)$ as distribution.
It is easy to check that $\Phi_{p}$ is Lipschitz continuous for the $2$-Wasserstein distance, uniformly in $n \geq 1$. We then consider the approximating sequence
$$
\hat{f}_{p} : \RR^d \times {\mathcal P}_{2}(\RR^d) \times \RR^k \ni (x,\mu,\alpha) \hookrightarrow 
f_{n_{p}}\bigl(x,\Phi_{p}(\mu),\alpha), \quad p \geq 1,
$$ 
instead of $(f_{n})_{n \geq 1}$ itself. Clearly, on any bounded subset, $\hat{f}_{p}$ still coincides with $f$ for $p$ large enough. Moreover, the conclusion of the second step is preserved. In particular, the conclusion of the second step together with \eqref{eq:5:5:15}, \eqref{eq:5:5:12}
 and \eqref{eq:5:5:16} say that 
(A.5) holds (for a possible new choice of $c_{L}$). From now on, we get rid of the symbol ``hat''
in $(\hat{f}_{p})_{p \geq 1}$
 and keep the notation 
$(f_{n})_{n \geq 1}$ for $(\hat{f}_{p})_{p \geq 1}$.
\vskip 4pt

\textit{Fourth Step.}
It only remains to check that $f_{n}$ satisfies the bound (A.6) and the sign condition (A.7). 
Since $\vert \partial_{\alpha} f(x,\mu,0) \vert \leq c_{L}$, the Lipschitz property of
$\partial_{\alpha} f$ implies that there exists a constant $c \geq 0$ such that 
$\vert \partial_{\alpha} f(x,\mu,\alpha) \vert \leq c$ for all $(x,\mu,\alpha) \in \RR^d \times {\mathcal P}_{2}(\RR^d) \times \RR^k$ with $\vert \alpha \vert \leq 1$. In particular, for any $n \geq 1$, it is plain to see that 
$
f_{n}(x,\mu,\alpha) \leq f_{n}(x,\mu,0) + c \vert \alpha \vert,
$
for any $(x,\mu,\alpha) \in  \RR^d \times {\mathcal P}_{2}(\RR^d) \times \RR^k$ with $\vert \alpha \vert \leq 1$, so that $\vert \partial_{\alpha} f_{n}(x,\mu,0) \vert \leq c$. This proves (A.6). 

Finally, we can modify the definition of $f_{n}$ once more to satisfy (A.7). Indeed, for any $R>0$, there exists an integer $n_{R}$, such that, for any $n \geq n_{R}$, 
$f_{n}(x,\mu,\alpha)$ and $f(x,\mu,\alpha)$ coincide for $(x,\mu,\alpha) \in \RR^d \times {\mathcal P}_{2}(\RR^d) \times \RR^k$ with $\vert x \vert, \vert \alpha \vert,M_{2}(\mu) \leq R$ so that 
$
\langle x,\partial_{x} f_{n}(0,\delta_{x},0) \rangle \geq -c_{L} (1+ \vert x \vert),
$
for  $\vert x \vert \leq R$ and $n \geq n_{R}$.
Next we choose a smooth function $\psi: \RR^d \hookrightarrow \RR^d$, satisfying $\vert \psi(x) \vert \leq 1$ for any $x \in \RR^d$, $\psi(x)=x$ for $\vert x \vert \leq 1/2$ and $\psi(x)= x/\vert x \vert$ for $\vert x \vert \geq 1$, and we set
$\hat{f}_{p}(x,\mu,\alpha) = f_{n_{p}}\bigl(x,\Psi_{p}(\mu),\alpha \bigr)$
for any integer $p \geq 1$ and $(x,\mu,\alpha) \in \RR^d \times {\mathcal P}_{2}(\RR^d) \times \RR^k$
where $\Psi_{p}(\mu)$ is the push-forward of $\mu$ by the mapping 
$
\RR^d \ni x \hookrightarrow x - \overline \mu  + p \psi( p^{-1} \langle \mu \rangle)$. Recall that
$\overline \mu$ stands for the mean of $\mu$.  
In other words, if $X$ has distribution $\mu$, then
$
\hat{X}_{p} = X - \EE(X) + p \psi ( p^{-1} \EE(X) )
$
has  distribution $\Psi_{p}(\mu)$.

$\Psi_{p}$ is Lipschitz continuous with respect to $W_{2}$, uniformly in $p \geq 1$. Moreover, for any $R>0$ and $p \geq 2R$, $M_{2}(\mu) \leq R$ implies 
$\vert \int_{\RR^d} x' d\mu(x') \vert
\leq R$ so that 
$p^{-1} \vert \int_{\RR^d} x' d\mu(x') \vert \leq 1/2$, that is $\Psi_{p}(\mu) = \mu$ and, for $\vert x \vert, \vert \alpha \vert \leq R$,   
$
\hat{f}_{p}(x,\mu,\alpha) = f_{n_{p}}(x,\mu,\alpha) = f(x,\mu,\alpha)$. 
Therefore, the sequence $(\hat{f}_{p})_{p \geq 1}$ is an approximating sequence for $f$ which satisfies the same regularity properties as $(f_{n})_{n \geq 1}$.
In addition,
$$
\langle x,\partial_{x} \hat{f}_{p}(0,\delta_{x},0) \rangle 
= \langle x,\partial_{x} f_{n_{p}}(0,\delta_{p \psi(p^{-1}x)},0) \rangle
= \langle x,\partial_{x} f(0,\delta_{p \psi(p^{-1}x)},0) \rangle 
$$
for $x \in \RR^d$. Finally we choose $\psi(x)=[\rho(\vert x \vert)/\vert x \vert]x$ (with $\psi(0)=0$), where $\rho$ is a smooth non-decreasing function from $[0,+\infty)$ into $[0,1]$ such that $\rho(x)=x$ on $[0,1/2]$ and $\rho(x)=1$ on $[1,+\infty)$. If $x \not =0$, then the above right-hand side is equal to 
\begin{equation*}
\begin{split}
\langle x,\partial_{x} f(0,\delta_{p \psi(p^{-1}x)},0) \rangle 
&= 
\frac{\vert p^{-1} x \vert}{ \rho(\vert p^{-1} x \vert) }
\langle p \psi(p^{-1}x),\partial_{x} f(0,\delta_{p \psi(p^{-1}x)},0) \rangle 
\\
&\geq - c_{L}
\frac{\vert p^{-1} x \vert}{ \rho(\vert p^{-1} x \vert) } \bigl( 1 + \vert 
p \psi(p^{-1} x) \vert \bigr). 
\end{split}
\end{equation*}
For $\vert x \vert \leq p/2$, we have $\rho(p^{-1} \vert x \vert) = \vert p^{-1}x \vert$, so that the right-hand side coincides with 
$-c_{L}(1+ \vert x \vert)$. For $\vert x \vert \geq p/2$, we have $\rho(p^{-1} \vert x \vert) \geq 1/2$ so that  
$$
- 
\frac{\vert p^{-1} x \vert}{ \rho(\vert p^{-1} x \vert) } \bigl( 1 + \vert 
p \psi(p^{-1} x) \vert \bigr) \geq -2  p^{-1} \vert x \vert \bigl( 1 + \vert 
p \psi(p^{-1} x) \vert \bigr)
\geq -2 p^{-1} \vert x \vert \bigl( 1 + p \bigr) \geq - 4 \vert x \vert.
$$
This proves that (A.7) holds with a new constant. 
$\Box$ 

\bibliographystyle{plain}

\end{document}